\documentclass{amsart}
\usepackage{xypic,tikz}
\usepackage{amssymb,
amsmath,
amsthm,
graphicx,
amscd,
multind,
eufrak,
hyperref,
}
\theoremstyle{plain}
\newtheorem{thm}{Theorem}[section]
\newtheorem{lm}[thm]{Lemma}
\newtheorem{cor}[thm]{Corollary}
\newtheorem{prop}[thm]{Proposition}
\newtheorem{lem}[thm]{Lemma}
\theoremstyle{definition}
\newtheorem{de}[thm]{Definition}
\newtheorem{ex}[thm]{Example}
\newtheorem{re}[thm]{Remark}

\newcommand{\RR}{{\mathbb R}}
\newcommand{\QQ}{{\mathbb Q}}
\newcommand{\ZZ}{{\mathbb Z}}
\newcommand{\ZZz}{{\mathbb Z}_{\geq 0}}
\newcommand{\NN}{{\mathbb N}}
\newcommand{\PP}{{\mathbb P}}
\newcommand{\pafg}[2][]{\frac{\partial #1}{\partial #2}}
{\begin{figure} \begin{center}}%
{\end{center} \end{figure}}

\def\sgn{{\rm sgn}\,}

\newcommand{\Sym}{\operatorname{Sym}}
\newcommand{\rk}{\operatorname{rk}}
\newcommand{\Gr}{\operatorname{Gr}\nolimits}
\newcommand{\bX}{\mathbf{X}}
\newcommand{\bY}{\mathbf{Y}}

\newcommand{\la}{\langle}
\newcommand{\ra}{\rangle}
\newcommand{\Wedge}{\bigwedge\nolimits}

\newcommand{\GL}{\operatorname{GL}\nolimits}

\newcommand{\biw}{{\bigwedge}}
\newcommand{\pv}{{Pl\"{u}cker variety}}
\newcommand{\pvs}{{Pl\"{u}cker varieties}}
\newcommand{\plg}{\operatorname{\mathbf{X}}}  %general Plücker variety
\newcommand{\X}{\operatorname{X}} %Plücker variety with $V_{n,p}$

\newcommand{\Pf}{\operatorname{Pf}}   %Pfaffian variety
\newcommand{\dirwed}{\biw^{\frac{\infty}{2}}V_{\infty}}
\newcommand{\Xinf}{\X_{\infty}}
\newcommand{\Yinf}{Y_{\infty}}
\newcommand{\Ginf}{G_{\infty}}

\newcommand{\Projd}{\pi}
\newcommand{\rem}{\ar@{^{(}->}[r]}
\newcommand{\dem}{\ar@{^{(}->}[d]}
\newcommand{\lpr}{\ar@{->>}[l]}
\newcommand{\upr}{\ar@{->>}[u]}

\newcommand{\p}{p}

\newcommand{\Mat}{{\operatorname{Mat}}}
\newcommand{\Matp}{{(\Mat_{\NN,\NN})^{\p}}}   %Set of p-tuples of infinite by infinite matrices
\newcommand{\ma}{M}                        %elt of \Matp

\newcommand{\Matfin}{{(\Mat_{N_1,N_2})^{\p}}}

\newcommand{\Colfin}{\Mat_{N_1,n}}
\newcommand{\Colp}{\Mat_{\NN,n}}              %Set of n-tuples of infinite columns
\newcommand{\Rowfin}{\Mat_{m,N_2}}
\newcommand{\Rowp}{\Mat_{m,\NN}}              %Set of n-tuples of infinite rows
\newcommand{\Fin}{K^d}

\newcommand{\Tot}{A}                        %Total space
\newcommand{\x}{x}                          %elt of \Tot
\newcommand{\xma}{x_{\mathrm{ma}}}          %elt of \Tot
\newcommand{\xcol}{x_{\mathrm{col}}}          %elt of \Tot
\newcommand{\xrow}{x_{\mathrm{row}}}          %elt of \Tot
\newcommand{\Totfin}{\Tot_{\mathrm{fin}}}

\newcommand{\GLn}{\GL_\NN}                %'negative' automorphisms
\newcommand{\GLp}{\GL_\NN}                 %'positive' automorphisms
\newcommand{\GLprod}{\GLn\times\GLp}

\begin{document}

\title[Pl\"ucker varieties and higher secants of Sato's Grassmannian]{Pl\"ucker varieties and higher secants\\
of Sato's Grassmannian}

\begin{abstract}
Every Grassmannian, in its Pl\"ucker embedding, is defined by quadratic
polynomials. We prove a vast, qualitative, generalisation of this
fact to what we call {\em Pl\"ucker varieties}. A Pl\"ucker variety
is in fact a family of varieties in exterior powers of vector spaces
that, like the Grassmannian, is functorial in the vector space and
behaves well under duals. A special case of our result says that for
each fixed natural number $k$, the $k$-th secant variety of {\em
any} Pl\"ucker-embedded Grassmannian is defined in bounded degree
independent of the Grassmannian. Our approach is to take the limit of
a Pl\"ucker variety in the dual of a highly symmetric space known as
the {\em infinite wedge}, and to prove that up to symmetry the limit
is defined by finitely many polynomial equations. For this we prove
the auxilliary result that for every natural number $p$ the space
of $p$-tuples of infinite-by-infinite matrices is Noetherian modulo
row and column operations. Our results have algorithmic counterparts:
every bounded Pl\"ucker variety has a polynomial-time membership test,
and the same holds for Zariski-closed, basis-independent properties of
$p$-tuples of matrices.
\end{abstract}

\author[J.~Draisma]{Jan Draisma}
\address[Jan Draisma]{
Department of Mathematics and Computer Science\\
Technische Universiteit Eindhoven\\
P.O. Box 513, 5600 MB Eindhoven, The Netherlands;
and Vrije Universiteit and Centrum voor Wiskunde en Informatica, Amsterdam,
The Netherlands}
\thanks{Both authors are supported by the first author's
Vidi grant from
the Netherlands Organisation for Scientific Research (NWO)}
\email{j.draisma@tue.nl}

\author[R.H.~Eggermont]{Rob H. Eggermont}
\address[Rob H. Eggermont]{
Department of Mathematics and Computer Science\\
Technische Universiteit Eindhoven\\
P.O. Box 513, 5600 MB Eindhoven, The Netherlands}
\email{r.h.eggermont@tue.nl}

\maketitle

%\tableofcontents

\section{Introduction and main results} \label{sec:intro}
To motivate our results, we first recall Grassmannians in their
Pl\"ucker embeddings.  For a natural number $p$ and a vector space $V$
over a field $K$, the Grassmannian of $p$-dimensional subspaces of $V$
lives in the projective space associated to the $p$-th exterior power
$\biw^p V$ of $V$. Let $\Gr(p,V) \subseteq \biw^p V$ denote the affine
cone over that Grassmannian. It consists of all {\em pure} tensors, i.e.,
those of the form $v_1 \wedge \cdots \wedge v_p$ with each $v_i \in V$.

As $p$ and $V$ vary, the varieties $\Gr(p,V)$ satisfy two fundamental
axioms. First, if $f: V \to W$ is a linear map, then the induced linear
map $\biw^pf: \biw^pV \to \biw^pW$ maps $\Gr(p,V)$ into $\Gr(p,W)$.
In particular, $\Gr(p,V)$ is stable under linear automorphisms of $V$.
Second, if $V$ has dimension $p+n$ with $p,n \geq 0$, then the natural
linear isomorphism $\biw^pV \to \biw^nV^*$ (natural, that is, up to a
scalar) maps $\Gr(p,V)$ into $\Gr(n,V^*)$. Indeed, its projectivisation
maps a point in the first Grassmannian, representing a $p$-dimensional
subspace $U$ of $V$, to the point in the second Grassmannian that
represents the annihilator $U^0$ of $U$ in $V^*$.

In this paper, we consider a general family $\{\plg_p(V) \subseteq
\biw^p V\}_{p,V}$ of closed subvarieties of exterior powers satisfying
the same two axioms. We call such a family a {\em \pv{}}; see
Section~\ref{sec:fincase} for a formal definition.  Thus a \pv{} is
not a single variety but rather a rule $\plg$ that assigns to a number
$p$ and a finite-dimensional $K$-vector space $V$ a closed subvariety
$\plg_p(V) \subseteq \Wedge^p V$, subject to the axioms above. A \pv{}
is called {\em bounded} if $\plg_2(V) \subsetneq \Wedge^2 V$ for at least
some finite-dimensional vector space $V$ (and hence, as we will see,
for all $V$ of sufficiently high dimension).

To avoid the anomaly that the Zariski topology becomes discrete, we
will assume throughout that the ground field $K$ of our vector spaces
and varieties is infinite. Our main theorem is then as follows.

\begin{thm}[Main Theorem] \label{thm:main1}
For any bounded Pl\"ucker variety $\plg$ there exists a $p_0 \in \ZZz$
and a finite-dimensional vector space $V_0$ such that all instances
$\plg_p(V)$ of $\plg$ are defined set-theoretically by polynomial
equations obtained from those of $\plg_{p_0}(V_0)$ by pulling back
along sequences of linear maps of the two types above. In particular,
$\plg_p(V)$ is defined set-theoretically by equations of bounded degree.
\end{thm}

As we will see below, the class of bounded Pl\"ucker varieties is
closed under taking secant varieties. So a direct consequence of our
main theorem is the following.

\begin{cor}
For any $k$, there exists a $d$ such that for any natural number $p$
and any finite-dimensional $K$-vector space $V$, the $k$-th secant variety of
$\Gr(p,V)$ inside $\Wedge^p V$ is defined set-theoretically by polynomials
of degree at most $d$.
\end{cor}

For instance, consider the Grassmannian itself. It is well known that
$\Gr(p,V)$ is defined by certain equations of degree two called Pl\"ucker
relations. But in fact, up to coordinate changes a {\em single} Pl\"ucker
relation suffices. Indeed, a set of defining equations for $\Gr(p,V)$
can be found simply by taking pullbacks of the Klein quadric defining
$\Gr(2,K^4)$ \cite{kprs}.  In general, this set does not generate
the full ideal of the Grassmannian. But it does show that one can test
membership of $\Gr(p,V)$ using a single type of equation; our result
above generalises this statement to Pl\"ucker varieties. It has the
following algorithmic consequence.

\begin{thm} \label{thm:main2}
For any bounded Pl\"ucker variety $\plg$ there exists a polynomial-time
algorithm that on input $d,p \in \ZZz$ and $\omega \in \biw^p K^d$
tests whether $\omega \in \plg_p(K^d)$.
\end{thm}

Here $\omega$ is given in a non-sparse encoding, and polynomial-time
refers to the number of arithmetic operations over the field generated by
the entries of $\omega$ (this field, unlike $K$, may be finite). Moreover,
if that field is the rational numbers or a more general number field,
then even the bit-complexity of the algorithm is polynomial.

\begin{re}
We stress that in Theorem~\ref{thm:main2} the Pl\"ucker variety $X$
is fixed (for instance, equal to {\em the third secant variety of the
Grassmannian}) and that the proof of the theorem is non-constructive in
the sense that the $p_0$ and $V_0$ from Theorem~\ref{thm:main1} have
to be built into the algorithm. We do not claim that $p_0$ and $V_0$
themselves can be found efficiently.
\end{re}

This paper is motivated by two main goals. The first is to understand
varieties built up from Grassmannians by operations such as joins, secant
varieties, and tangential varieties. All of these are examples of bounded
Pl\"ucker varieties. For instance, if $\plg,\bY$ are bounded Pl\"ucker
varieties, then the rule $\plg+\bY$ that assigns to $p,V$ the Zariski
closure of $\{x+y \mid x \in \plg_p(V), y \in \bY_p(V)\} \subseteq
\biw^pV$ is again a bounded Pl\"ucker variety, called the {\em join}
of $\plg$ and $\bY$, to which our theorem applies. In the special case
where $\bY=\plg$, the join is called the {\em secant variety} of $\plg$,
and higher secant varieties are obtained by repeatedly taking the join
with $\plg$. Quite a bit is known about the {\em dimensions} of these
secant varieties \cite{Catalisano05,Baur06}, but almost nothing is known
about their defining equations.

%Our main theorem implies that for each
%fixed $k$, the $k$-th higher secant variety of {\em any} Grassmannian
%$\Gr(p,V)$ is defined by polynomials of bounded degree uniformly in $p$
%and $V$. This statement is new for all $k>1$ (where $k=1$ corresponds
%to the Grassmannian itself).

The second goal is to develop an exterior-power analogue of Snowden's
theory of $\Delta$-varieties \cite{Snowden13}. That theory concerns
varieties (or schemes) of ordinary tensors, rather than alternating
tensors.  For ordinary tensors, the analogues of our results are
established in \cite{Draisma11d}. Also, for ordinary tensors, many
more concrete results are known on equations for the first few higher
secant varieties \cite{Strassen83,Landsberg04,Raicu10,Qi13}. For {\em
symmetric tensors}, also quite a lot is known about equations for secant
varieties (see, e.g., \cite{Buczynska13,Landsberg11}), but we do not yet
know whether a symmetric counterpart to our Theorems~\ref{thm:main1}
and~\ref{thm:main2} exists.

The proofs of both theorems are non-constructive. In particular, we do not
find new equations for secant or tangential varieties of Grassmannians
other than pullbacks of Pfaffians. Finding explicit equations is an
art that involves sophisticated techniques from representation theory
\cite{Landsberg11,Manivel14}.  Instead, we will establish the fundamental
fact that up to symmetry, finitely many equations suffice.

The key notion in our approach is {\em Noetherianity up to symmetry}.
A topolo\-gical
space on which a group $G$ acts by means of homeomorphisms is said to
be $G$-Noetherian (or equivariantly Noetherian if $G$ is clear from
the context) if every descending chain of closed, $G$-stable subsets
stabilises. We refer to~\cite[Chapter 2]{Conca14} for a gentle
introduction. There is currently a surge of activity on related
stabilisation issues in algebraic geometry and its applications,
e.g. in algebraic statistics \cite{Hillar09,Draisma08b,Draisma12f},
invariant theory~\cite{Howard09}, representation theory~\cite{Church12},
and commutative algebra~\cite{Sam12,SamSnow12,DEKL13}. The following new
example of this phenomenon will play a fundamental role in the proofs
of our theorems, but is likely to be useful in other applications. Let
$\Mat_{\NN,\NN}$ denote the (uncountably dimensional) space of all
$\NN \times \NN$-matrices over $K$. Similarly, for $n,m \in \ZZz$ define
$\Mat_{\NN,n}$ and $\Mat_{m,\NN}$. Consider the group $\GL_\NN:=\bigcup_{n
\in \NN} \GL_n$ of all invertible matrices having zeroes almost
everywhere outside the diagonal (i.e. everywhere outside the diagonal
except in a finite number of positions) and ones almost everywhere on
the diagonal. One copy of this group acts by left multiplication on
$\Mat_{\NN, \NN}$ and $\Mat_{\NN,n}$ and trivially on $\Mat_{m,\NN}$,
and one copy acts by right multiplication on $\Mat_{\NN,\NN}$ and
$\Mat_{m,\NN}$ and trivially on $ \Mat_{\NN,n}$. For any $p,d \in \ZZz$,
consider the Cartesian product
\[
\Tot_{p,n,m,d}:=\Matp \times \Mat_{\NN,n} \times \Mat_{m,\NN} \times K^d,
\] equipped with the Zariski topology in which closed sets are given by
polynomials in the entries of the $p+2$ matrices and the coordinates
on the latter $K^d$. Let $\GL_\NN \times \GL_\NN$ act diagonally, and
trivially on $K^d$.

\begin{thm} \label{thm:main3}
For any $p,n,m,d \in \ZZz$, the topological space
$\Tot_{p,n,m,d}$ is equivariantly Noetherian with respect to
$\GL_\NN \times \GL_\NN$. In other words, every $\GL_\NN \times
\GL_\NN$-stable closed subset can be characterised as the common zero
set of finitely many $\GL_\NN \times \GL_\NN$-orbits of polynomial equations.
\end{thm}

We do not know whether the corresponding ideal-theoretic statement
also holds, i.e., whether each $\GL_\NN \times \GL_\NN$-stable ideal
in the coordinate ring of the variety in the theorem is generated by
finitely many orbits of polynomials. We return to this question in
Section~\ref{sec:disc}.

The remainder of this paper is organised as follows. In
Section~\ref{sec:fincase}, we give a formal definition of \pvs{} and
discuss the boundedness condition.  In Section~\ref{sec:infwedge}
we introduce the {\em infinite wedge} (or rather, its charge-zero
part), and in Section~\ref{sec:limit} we construct, for any \pv{},
a limit object in the space dual to the infinite wedge. In the case
of the Grassmannian, this limit is known as (the charge-zero part
of) {\em Sato's Grassmannian} \cite{Sato83,Segal85,Alvarez98}. For
a general bounded \pv{}, the limit lies in the variety defined by
certain Pfaffians, which we describe in Section~\ref{sec:pfaff}. Then,
in Section~\ref{sec:noeth} we show that these {\em Pfaffian varieties}
are equivariantly Noetherian with respect to a group that preserves
the limit of any \pv{}. In particular, this shows that the limit
is defined by finitely many equations under that group. Finally, in
Section~\ref{sec:backfin}, we go back to finite-dimensional instances
of a \pv{} and complete the proof of Theorems~\ref{thm:main1}
and~\ref{thm:main2}. The Noetherianity in Section~\ref{sec:noeth}
is proved using the auxilliary Theorem~\ref{thm:main3}, whose proof we
defer to Section~\ref{sec:mxtuples}. Finally, in Section~\ref{sec:disc},
we discuss a number of open questions.

\section{Pl\"ucker varieties and boundedness}\label{sec:fincase}
Throughout this paper, we work over an infinite field $K$, and we use
the convention $\NN = \{1,2,3,\ldots\}$ and $[n]:=\{1,\ldots,n\}$ for
$n \in \ZZz$. By a (finite-dimensional) variety we mean a Zariski-closed
subset of a finite-dimensional vector space over $K$. Sometimes we will
stress the closedness and say {\em closed subvariety}.

If $V$ is a vector space with basis $v_1,\ldots,v_m$, then $\Wedge^p
V$ has a basis consisting of the vectors $v_I:=v_{i_1} \wedge \cdots
\wedge v_{i_p}$ where $I=\{i_1<\ldots<i_p\}$ runs over all $p$-subsets of
$[m]$. We will call this the {\em standard basis} of $\Wedge^p
V$ relative to the given basis $\{v_i\}_i$.  We always identify $(\biw^p
V)^*$ with $\biw^p (V^*)$ via the map from the latter space to the
former that sends $x_1 \wedge \cdots \wedge x_p$ to the linear function
determined by
\[ v_1\wedge \cdots \wedge v_p \mapsto \sum_{\pi \in S_p} \sgn(\pi) \prod_{i=1}^p
x_i(v_{\pi(i)}). \]
Given a vector space $V$ of dimension $p+n$ with $p,n \in
\ZZz$, and choosing an isomorphism $\psi: \biw^{p+n}V \to K$,
we obtain an isomorphism $\star: \biw^pV \to \biw^nV^*$ defined by
$\star(\omega)(\omega') = \psi(\omega \wedge \omega')$. The map $\star$
is well-defined up to choice of $\psi$. We call such a map a \emph{Hodge
dual}. Classically, Hodge duals go one step further in identifying $\biw^n
V^*$ with $\biw^n V$ by means of a symmetric bilinear form on $V$,
but we will not do this.  Note that if $\star: \biw^pV \to \biw^nV^*$
is a Hodge dual, then so is its inverse $\star^{-1}:
\biw^nV^* \to \biw^pV$ and its dual $\star^* : \biw^n V \to
\biw^p V^*$.

\begin{de} \label{de:pv}
A \emph{\pv{}} is a sequence $\plg =
(\plg_p)_{p\in\ZZz}$ of functors from the category of
finite-dimensional vector spaces to the category of
varieties satisfying the following axioms:
\begin{enumerate}
\item For all vector spaces $V$ and for all $p \in \ZZz$,
the variety $\plg_p(V)$ is a closed subvariety of $\biw^pV$.
\item For all $p \in \ZZz$ and for all linear maps $\varphi: V
\to W$, the map $\plg_p(\varphi): \plg_p(V) \to \plg_p(W)$ is the restriction of $\biw^p\varphi$ to $\plg_p(V)$.
\item If $V$ is a vector space of dimension $n+p$ with
$n,p\in\ZZz$, and $\star: \biw^pV \to \biw^nV^*$ is a Hodge
dual, then $\star$ maps $\plg_p(V)$ into $\plg_n(V^*)$.
\end{enumerate}
Given a \pv{} $\plg$, a variety of the form $\plg_p(V)$
for a specific choice of $p$ and $V$ is called an {\em
instance} of the \pv{} $\plg$.
\end{de}

\begin{ex} \label{ex:constructions}
The following constructions give a rich source of \pvs.
\begin{enumerate}
\item $\plg_p(V):=\Wedge^p V$, $\plg_p(V):=\emptyset$, $\plg_p(V):=\{0\}$
are \pvs{}.
\item $\plg_p(V):=\Gr(p,V)$ is the (cone over the) Grassmannian; we will
see that this is the smallest non-zero \pv{}.
\item Given \pvs{} $\plg$ and $\bY$, the rules $\plg \cap \bY$ and $\plg
\cup \bY$ defined in the obvious manner are \pvs{}; and
\item Similarly, the {\em join} $\plg+\bY$ and
the {\em tangential variety} $\tau \plg$ defined by
\begin{align*} &(\plg+\bY)_p(V):=\overline{\{x+y \mid x \in \plg_p(V), y
\in \bY_p(V)\}} \text{ and}\\
&(\tau \plg)_p(V):=\overline{\{x \mid x \in \ell \text{ for some line } \ell
\text{ tangent to } \bX_p(V) \text{ at a smooth
point}\}}
\end{align*}
are \pvs{}.
\hfill $\diamondsuit$
\end{enumerate}
\end{ex}

The second axiom implies that for any \pv{}, any $p \in \ZZz$ and any $V$,
the instance $\plg_p(V)$ is stable under the action of $\GL(V)$. Moreover,
it is stable under multiplication with scalars: if $\star: \biw^pV
\to \biw^nV^*$ is a Hodge dual, then for any scalar $t \in K$, we
have $t\plg_p(V) = (t\star^{-1})\star \plg_p(V) \subseteq \plg_p(V)$,
because $t\star^{-1}$ is a Hodge dual, as well. Hence, provided that it is
non-empty, $\plg_p(X)$ is the affine cone over a projective variety. To
avoid having to deal with rational maps, we work with the cone rather
than the projective variety.

When combined, the axioms for \pvs{} give further maps connecting
instances of $\plg$. The following lemmas extract two fundamental types
of such maps.

\begin{lem}[Tensoring] \label{lm:augmentation}
Let $W$ be a finite-dimensional vector space, let $V$ be a codimension-one
subspace of $W$, and write $W=V \oplus \langle w \rangle$. Then for any
\pv{} $\plg$ and any $p \in \ZZz$ with $p \leq \dim V$ the map
\[ \biw^p V \to \biw^{p+1} W,\quad \omega \mapsto \omega \wedge w  \]
maps $\plg_p(V)$ into $\plg_{p+1}(W)$.
\end{lem}

\begin{proof}
Set $n:=\dim V-p$. The map in the lemma is the composition
\begin{equation} \label{eq:littlediag}
\xymatrix{
\Wedge^p V \ar[r]^{\star_1} &
\Wedge^n V^* \ar[r]^{\biw^n \varphi} &
\Wedge^n W^* \ar[r]^{\star_2} &
\Wedge^{p+1} W
}
\end{equation}
where $\star_1,\star_2$ are suitable Hodge duals and $\varphi: V^* \to W^*$
is the map that extends a linear function by zero on $\langle w \rangle$.
\end{proof}

\begin{lem}[Contraction] \label{lm:contraction}
In the setting of the previous lemma, let $\iota$ denote the
embedding $V \to W$, so that $\iota^*:W^* \to V^*$ is
restriction of linear functions. Then
the linear map determined by
\begin{align*}
\biw^{p+1} W^* &\to \biw^p V^*\\
x_1 \wedge \cdots \wedge x_{p+1} &\mapsto
\sum_{i=1}^{p+1} (-1)^{p+1-i} x_i(w) \cdot
(\Wedge^p \iota^*)(x_1 \wedge \cdots \widehat{x_i} \cdots
\wedge x_{p+1})
\end{align*}
maps $\plg_{p+1}(W^*)$ into $\plg_p(V^*)$.
\end{lem}

\begin{proof}
First, we claim that this map is the dual of the map in the previous
lemma. Indeed, evaluating $x_1 \wedge \cdots \wedge x_{p+1}$ on $v_1
\wedge \cdots \wedge v_p \wedge w$ yields the same as evaluating the
right-hand side above on $v_1 \wedge \cdots \wedge v_p$.  Now the
desired result follows by taking the dual maps in Diagram
\eqref{eq:littlediag}.
\end{proof}

Here is a first illustration of how a single instance of a
\pv{} may determine all of it.

\begin{lem} Let $\plg$ be a \pv{}. Then  $\plg_0(K) = \{0\}$ if and only
if $\plg_p(V) = \{0\}$ for all $p$ and $V$.
\end{lem}

\begin{proof}
The implication $\Leftarrow$ is immediate. For the
implication $\Rightarrow$,
pick a non-zero vector $\omega \in \plg_p(V)$, let $v_1,\ldots,v_p$ be a
basis of $V$, and assume that the coefficient in $\omega$ of the standard
basis vector $v_{i_1} \wedge \cdots \wedge v_{i_p}$ is non-zero. Set
$W:=\la v_{i_1},\ldots,v_{i_p} \ra$ and let $\pi:V \to W$ be the
projection along the remaining basis vectors. Then $\omega':=(\Wedge^p
\pi) \omega$ is a non-zero element of $\plg_p(W)$. Next, apply
$\star:\Wedge^p W \to \Wedge^0 W^*$ to $\omega'$ to obtain a non-zero
$\omega'' \in \plg_0(W^*)$. Finally, the zeroth exterior power of any
linear map $W^* \to K$ is an isomorphism and maps $\omega''$ to a non-zero
element of $\plg_0(K)$.
\end{proof}

It follows that if $\plg_0(V) = \{0\}$ for some $V$, then the full \pv{}
is zero. Moreover, since the only closed subvarieties of $V$ that are
$\GL(V)$-stable are $V$ and $\{0\}$, we find that if $\plg_1(V) \neq V$
for some $V$, then $\plg_1(V) = \{0\}$ and consequently, the full \pv{} is
zero. So only \pvs{} with $\plg_1(V)=V$ for all $V$ are of interest to us.

Next, $\GL(V)$ has exactly $\lfloor \frac{\dim V}{2} \rfloor$ orbits
on $\biw^2 V$. Indeed, any $\omega$ in this space is of the form $v_1
\wedge v_2 + \ldots + v_{2r-1} \wedge v_{2r}$ with $v_1,\ldots,v_{2r}$
linearly independent, so that $2r \leq \dim V$. The number $r$ is called
the {\em rank} of $\omega$, denoted $\rk \omega$. It is half the rank
of the skew-symmetric matrix $((x_i \wedge x_j)(\omega))_{ij}$ where $x_i,x_j$
range over a basis of $V^*$. For $r$ in this range, define
\[ Y^r(V):=\{ \omega \mid \rk(\omega) \leq r \}. \]
As we assume that $K$ is infinite, the $Y^r(V)$ are the only
Zariski-closed, $\GL(V)$-stable subsets of $\Wedge^2 V$.

\begin{lem}
Let $\plg$ be a \pv{}. Suppose that there exists a vector space $V$ such
that $\plg_2(V) = Y^r(V) \neq \biw^2V$. Then for all vector spaces $W$,
we have $\plg_2(W) = Y^r(W)$.
\end{lem}

\begin{proof}
Suppose that $\plg_2(W)$ contains $\omega$s with rank strictly
exceeding $r$. Since it is closed and $\GL(W)$-stable, it contains an $\omega$ of rank
equal to $r+1$. Write $\omega = w_1\wedge w_2 + \ldots + w_{2r+1}\wedge
w_{2r+2}$.  Note that $V$ has dimension at least $2r+2$, because
$Y^r(V) \neq \biw^2V$. Let $\varphi: W \to V$ be a linear map that maps
$w_1,\ldots,w_{2r+2}$ to linearly independent elements. Then
$\biw^2\varphi(\omega)$ has rank $r+1$. This gives a contradiction, since
$\biw^2\varphi(\omega) \in \plg_2(V)$. We conclude $\plg_2(W) \subseteq
Y^r(W)$.

Conversely, let $\omega \in Y^r(W)$ and write $\omega = w_1\wedge
w_2+\ldots + w_{2r'-1}\wedge w_{2r'}$ for some $r' \leq r$, with
$w_1,\ldots,w_{2r'}$ linearly independent.  Let $v_1,\ldots,v_{2r'} \in
V$ be linearly independent, and let $\varphi: V \to W$ be a linear map that
maps $v_i$ to $w_i$. We have $v_1\wedge v_2 + \ldots + v_{2r'-1}\wedge
v_{2r'} \in \plg_2(V)$, and its image under $\biw^2\varphi$ is $\omega$,
hence $\omega \in \plg_2(W)$. We conclude $\plg_2(W) = Y^r(W)$.
\end{proof}

Dually, define $Y^{r,\star}(V) := \star Y^r(V^*) \subseteq \biw^{\dim
V -2}V$. Note that $Y^{r,\star}(V)$ is independent of choice of Hodge
dual. By taking Hodge duals, we get matching statements in $\biw^{\dim V
-2}V$ for each $V$.

\begin{lem}\label{lem:starY} The only closed $\GL(V)$-stable
subvarieties of $\biw^{\dim V -2}V$ are the varieties
$Y^{r,\star}(V)$. Moreover, if $\plg$ is a \pv{}, and if $\plg_{\dim V -2}(V) = Y^{r,\star}(V) \neq \biw^{\dim V - 2}V$ for some $V$, then $\plg_{\dim W -2}(W) = Y^{r,\star}(W)$ for all $W$.
\end{lem}

\begin{de}
A \pv{} $\plg$ is called {\em bounded} if there exists some $V$ for which
$\plg_2(V) \neq \biw^2 V$. In this case, there exists a unique $r$ such
that $\plg_2(V)=Y^r(V)$ for all $V$, called the {\em rank} of the \pv{}.
\end{de}

By the above, the rank also satisfies $\plg_{\dim V-2}(V)=Y^{r,\star}(V)$
for all $V$ of dimension at least $2$. Note that the Grassmannian
is a bounded \pv{} of rank $1$, and that the constructions in
Example~\ref{ex:constructions} all preserve the class of bounded Pl\"ucker
varieties. For example, the rank of the join $\plg + \bY$ is at most
(and in fact equal to) the sum of the ranks of $\plg$ and $\bY$, and
the rank of the tangential variety $\tau \plg$, being contained in the
secant variety $\plg + \plg$, is at most twice the rank of $\plg$. This
shows that all \pvs{} of direct interest to us are bounded.

\section{The infinite wedge and its dual}
\label{sec:infwedge}

In this section we introduce the infinite wedge. We start with a
countable-dimensional vector space
\[ V_\infty=\langle \ldots,x_{-3},x_{-2},x_{-1},x_1,x_2,x_3,\ldots \rangle \]
in which the $x_i, i \in -\NN \cup \NN$ are a basis. Note that
we skip $0$, which makes the set-up symmetric around zero. In
the literature on the infinite wedge, this symmetry is achieved
by labelling with half-integers, and in {\em decreasing} order
\cite{Bloch00,Rios13}. Formulas from representation theory and integrable
systems depend on this convention. But we will not need any of those
formulas, so we take the liberty to simplify the notation and label with
$-\NN \cup \NN$ instead.

For any $n \in \ZZz$
(for ``negative'') and $p \in \ZZz$ (for ``positive'') let $V_{n,p}$
be the $(n+p)$-dimensional subspace
\[ V_{n,p}:= \langle x_{-n}, \ldots, x_{-2}, x_{-1}, x_1,
x_2, \ldots, x_p \rangle. \]
We arrange the exterior powers $\biw^p V_{n,p}$ into a
two-dimensional commutative diagram as follows.
\begin{equation} \label{eq:dirwedge}
\xymatrix{
\Wedge^0 V_{0,0} \rem \dem &
\Wedge^1 V_{0,1} \rem \dem &
\Wedge^2 V_{0,2} \rem \dem &
\ldots\\
\Wedge^0 V_{1,0} \rem \dem &
\Wedge^1 V_{1,1} \rem \dem &
\Wedge^2 V_{1,2} \rem \dem &
\ldots\\
\Wedge^0 V_{2,0} \rem \dem &
\Wedge^1 V_{2,1} \rem \dem &
\Wedge^2 V_{2,2} \rem \dem &
\ldots\\
\vdots & \vdots & \vdots &
}
\end{equation}
Here the vertical maps $\Wedge^p V_{n,p} \to \Wedge^p V_{n+1,p}$ are just
the $p$-th exterior powers of the embeddings $V_{n,p} \to V_{n+1,p}$,
and the horizontal maps $\Wedge^p V_{n,p} \to \Wedge^{p+1} V_{n,p+1}$
are given by $\omega \mapsto \omega \wedge x_{p+1}$. Note that both of
these maps are injective. We will always identify $\Wedge^p V_{n,p}$
with a subspace of $\Wedge^{p'} V_{n',p'}$ for any $n' \geq n$ and $p'
\geq p$ by means of the appropriate sequence of these maps.

If $x_I=x_{i_1} \wedge \cdots \wedge x_{i_p}$ is a standard basis
element of $\Wedge^p V_{n,p}$, then the subset $I=\{i_1<\ldots<i_p\}
\subseteq \{-n,\ldots,-1,1,\ldots,p\}$ has the property that the number
of negative elements of $I$ equals $p$ minus the number of positive
elements of $I$. Under the vertical map this property is preserved. Under
the horizontal map, $x_I$ is identified with $x_{I'}$ with $I'=I \cup
\{p+1\}$, and $I'$ again has the property that its number of negative
elements equals $p+1$ minus its number of positive elements.

\begin{de}
The {\em infinite wedge} is defined as
\[ \dirwed := \lim_{\substack{\longrightarrow \\ p,n}} \biw^p V_{n,p} =
\bigcup_{p,n} \biw^p V_{n,p}, \]
where the limit is taken of the directed system above. It comes with a
standard basis consisting of elements $x_I=x_{i_1} \wedge
x_{i_2} \wedge \cdots$ where $I=\{i_1<i_2<\ldots\}
\subseteq -\NN \cup \NN$ has the property that $i_k=k$ for $k \gg 0$;
this element is the image in $\dirwed$ of $x_{i_1} \wedge
\cdots \wedge x_{i_p}
\in \Wedge^p(V_{n,p})$ for any choice of $n \geq -i_1$ and
of $p$ such that $i_k=k$ for $k \geq p$.
\end{de}

In fact, in the existing literature on the infinite wedge, this limit
is called the {\em charge zero} part of the infinite wedge. The
full infinite wedge then arises by allowing $k-i_k$ to be any constant for
$k \gg 0$, and this constant is called the {\em charge} of $x_I$. We
will restrict ourselves to the charge-zero part.

The basis vectors $x_I$ of $\dirwed$ are in one-to-one correspondence with
the set of all {\em Young diagrams}, and this will be a useful visual
aid later on. This correspondence is well known; see for instance the
discussion of {\em Maya diagrams} and partitions in~\cite{Rios13}. To
find the Young diagram corresponding to $I=\{i_1<i_2<\ldots\}$, proceed
as follows. Draw the first quadrant $\RR_{\geq 0}^2$ with horizontal
axis subdivided into unit intervals labelled by $\NN$ and vertical
axis subdivided into unit intervals labelled by $-\NN$. Subdivide the
quadrant into diagonal strips labelled by $-\NN \cup \NN$ running from
the corresponding (horizontal or vertical) intervals in northeasterly
direction. Now draw the lattice path that starts high north on the
vertical axis and goes south in a strip corresponding to an $i \not
\in I$ and east in a strip corresponding to an $i \in I$. The fact that
$i_k=k$ for $k \gg 0$ ensures that the path ends up on the horizontal
axis. The region below the lattice path is a Young diagram, which
uniquely determines the lattice path and the set $I$. For an example
see Figure~\ref{fig:young}.

\begin{figure}
\begin{center}
\includegraphics[scale=.7]{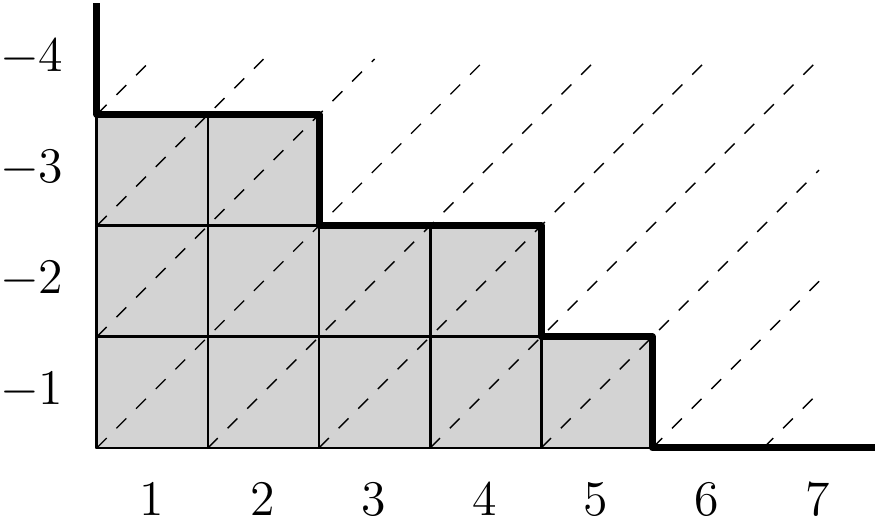}
\caption{Lattice path and Young diagram corresponding to
$I=\{-3,-2,1,2,4,6,7,8,\ldots\}$.}
\label{fig:young}
\end{center}
\end{figure}

The basis vectors $x_I$ have a natural partial order defined by $x_I
\preceq x_J$ if and only if $i_k \leq j_k$ for all $k$. This is equivalent
to the condition that the Young diagram corresponding to $I$ contains
the Young diagram corresponding to $J$. The unique largest element
has $I=\{1,2,3,\ldots\}$, and the partial order does not have infinite
strictly increasing chains. In fact, a much stronger statement holds:
the opposite partial order is a well-partial-order on the variables $x_I$
(and a similar statement holds for higher-dimensional partitions; see,
e.g., \cite{MacLagan01}), but we will not need this stronger statement.

Despite its apparent dependence on the choice of coordinates,
the infinite wedge has a large symmetry group acting on it. Indeed,
denote $G_{n,p} := \GL(V_{n,p})$ and embed $G_{n,p}$ into $G_{n+1,p}$
and $G_{n,p+1}$ by fixing the standard basis vector $x_{-(n+1)}$ and
$x_{p+1}$, respectively. Each of the two arrows emanating from $\Wedge^p
V_{n,p}$ is $G_{n,p}$-equivariant. As a consequence, the group $\Ginf
= \bigcup_{n,p}G_{n,p}$ acts on $\dirwed$. More explicitly, if an
element $g \in G_{n',p'}$ is to act on an element $\omega \in \Wedge^p
V_{n'',p''}$, then one sets $n:=\max\{n',n''\},p:=\max\{p',p''\}$, sees
$g$ as an element of $G_{n,p}$ and $\omega$ as an element in $\Wedge^p
V_{n,p}$, and performs the action there.

\begin{ex} \label{ex:derivations}
We note two consequences of the $\Ginf$-action on $\dirwed$ that will
become important later on.  First, embedding the symmetric group $S_{n+p}$
by means of permutation matrices into $G_{n,p}$, we find that the group
of all finitary permutations of $-\NN \cup \NN$, i.e., those that fix
all but a finite number of integers, acts on $\dirwed$. A finitary
permutation $\pi$ sends the basis vector $x_I$ to $\pm x_{\pi(I)}$,
where the sign depends on the number of pairs $i,j \in I$ with $i<j$
but $\pi(i)>\pi(j)$. All signed basis vectors
are contained in a single orbit under finitary permutations.

Second, the action of $\Ginf$ induces an action
of its Lie algebra by taking derivatives. This Lie algebra is spanned by the derivations
$\partial_{kl}:= x_k \cdot \pafg{x_l}$ as $k,l$ vary over $-\NN \cup \NN$. The
action of this derivation on a basis vector $x_I$ is obtained by writing
$x_I=x_{i_1} \wedge x_{i_2} \wedge \cdots$ and formally applying Leibniz'
rule. In other words,
\[
\partial_{kl} x_I =
\begin{cases}
	x_I & \text{if $k = l$ and $k \in I$,}\\
	\pm x_{I\setminus\{l\}\cup\{k\}} &
		\text{if $k \not \in I$ and $l \in I$, and}\\
	0 & \text{otherwise,}
\end{cases}
\]
with sign determined by the number of elements of $I$ strictly
between $k$ and $l$.\hfill $\diamondsuit$
\end{ex}	

\begin{re}
The symmetric algebra generated by the infinite wedge has an action
of $\Ginf$ by automorphisms. Since $\Ginf$ is isomorphic to the
infinite general linear group, one might think that this symmetric
algebra is a {\em twisted commutative algebra (tca)} in the sense of
\cite{Sam12,SamSnow12}. But this is not the case, since $\dirwed$ is not
a subquotient of any finite tensor power of the countably-dimensional
standard representation of the infinite general linear group. If one
restricts the attention to those $x_I$ with $I \supseteq \{p,p+1,\ldots\}$
for some fixed $p \in \NN$, acted on by the stabiliser in $\Ginf$ of
all $x_i$ for $i \geq p$, then one does obtain a tca. However, such a
tca is too small for our purposes. For instance, it
does not allow for proving statements about all
Grassmannians $\Gr(p,V)$ with both $p$ and $V$ varying.
\end{re}

In the following sections, we will be concerned with the {\em dual
infinite wedge} $(\dirwed)^*$. This uncountably-dimensional vector space
arises as the projective limit $\lim\limits_{\substack{\longleftarrow\\n,p}} \Wedge^p V_{n,p}^*$ of
the diagram obtained from Diagram \eqref{eq:dirwedge} by taking duals
of all arrows:
\begin{equation} \label{eq:projwedge}
\xymatrix{
\biw^0 V_{0,0}^* &
\biw^1 V_{0,1}^* \lpr &
\biw^2 V_{0,2}^* \lpr &
\ldots   \ar@{->>}[l]\\
\biw^0 V_{1,0}^* \upr &
\biw^1 V_{1,1}^* \lpr \upr &
\biw^2 V_{1,2}^* \lpr \upr &
\ldots \ar@{->>}[l]\\
\biw^0 V_{2,0}^* \upr &
\biw^1 V_{2,1}^* \lpr \upr &
\biw^2 V_{2,2}^* \lpr \upr &
\ldots \lpr \\
\vdots \upr & \vdots \upr & \vdots
\upr & \
}
\end{equation}
The symmetric algebra generated by $\dirwed$, which is
the polynomial ring in the $x_I$, serves as coordinate ring of the dual
infinite wedge. The dual infinite wedge carries the Zariski topology,
in which the closed subsets are those characterised by the vanishing of
a collection of polynomials in the $x_I$.

We will also need arrows going in the opposite direction.
For this, denote the basis of $V_{n,p}^*$ dual to the
standard basis by
$e_{-n},\ldots,e_{-1},e_1,\ldots,e_p$.
Take the $p$-th exterior power of the
embedding $V_{n,p}^* \to V_{n+1,p}^*,\ e_i \mapsto e_i$
as vertical maps and the map
\[ \Wedge^p V_{n,p}^* \to \Wedge^{p+1} V_{n,p+1}^*,\quad \omega
\mapsto \omega \wedge e_{p+1} \]
as horizontal map. These maps are right inverses (sections) of the
corresponding projections in Diagram~\eqref{eq:projwedge}, and they fit
into the commutative diagram
\begin{equation} \label{eq:dirwedgestar}
\xymatrix{
\Wedge^0 V^*_{0,0} \rem \dem &
\Wedge^1 V^*_{0,1} \rem \dem &
\Wedge^2 V^*_{0,2} \rem \dem &
\ldots\\
\Wedge^0 V^*_{1,0} \rem \dem &
\Wedge^1 V^*_{1,1} \rem \dem &
\Wedge^2 V^*_{1,2} \rem \dem &
\ldots\\
\Wedge^0 V^*_{2,0} \rem \dem &
\Wedge^1 V^*_{2,1} \rem \dem &
\Wedge^2 V^*_{2,2} \rem \dem &
\ldots\\
\vdots & \vdots & \vdots &
}
\end{equation}
In particular, this diagram allows us to lift an element of $\Wedge^{p}
V^*_{n,p}$ to an element in the dual infinite wedge. We will
return to this fact in Section~\ref{sec:backfin}.

\section{The limit of a Pl\"ucker variety} \label{sec:limit}

Let $\plg$ be a \pv{}, and evaluate $\X_{n,p} := \plg_p(V^*_{n,p})$. By the \pv{} axioms, the embedded variety
$\X_{n,p} \subseteq \Wedge^p V_{n,p}^*$ is the image of the embedded
variety $\plg_p(V) \subseteq \Wedge^p V$ for any $(n+p)$-dimensional
vector space $V$ under any isomorphism $V \to V^*_{n,p}$. Hence,
the $\X_{n,p}$ determine the \pv{} and they are stable under
$G_{n,p}=\GL(V_{n,p})$.

Next, the $X_{n,p}$ fit into two commutative diagrams
\begin{equation} \label{eq:diagX}
\xymatrix{
\X_{0,0} &
\X_{0,1} \lpr&
\X_{0,2} \lpr&
\ldots   \lpr\\
\X_{1,0} \upr &
\X_{1,1} \lpr \upr &
\X_{1,2} \lpr \upr &
\ldots \lpr \\
\X_{2,0} \upr &
\X_{2,1} \lpr \upr &
\X_{2,2} \lpr \upr &
\ldots \lpr \\
\vdots \upr & \vdots \upr & \vdots \upr & \
}
\quad
\xymatrix{
\X_{0,0} \rem \dem&
\X_{0,1} \rem \dem &
\X_{0,2} \rem \dem &
\ldots   \\
\X_{1,0} \rem \dem &
\X_{1,1} \rem \dem &
\X_{1,2} \rem \dem &
\ldots \\
\X_{2,0} \rem \dem &
\X_{2,1} \rem \dem &
\X_{2,2} \rem \dem &
\ldots \\
\vdots & \vdots & \vdots & \
}
\end{equation}
where the maps in the leftmost diagram are those from
Diagram~\ref{eq:projwedge} and those in the rightmost diagram are
those from Diagram~\ref{eq:dirwedgestar}.  The horizontal maps
preserve instances of \pvs{} because of Lemmas~\ref{lm:contraction}
and~\ref{lm:augmentation}, respectively. The vertical maps preserve
instances by Definition~\ref{de:pv}(2).

We denote $\Xinf := \lim\limits_{\longleftarrow}X_{n,p} \subseteq
(\dirwed)^*$, which we call the {\em limit} of the \pv{} $\plg$. This is
the subset of the dual infinite wedge consisting of all $\omega$ with
the property that for each $n,p$ the image of $\omega$ in $\Wedge^p
V_{n,p}^*$ lies in $\X_{n,p}$.  Equivalently, it is the zero set of
the union of all ideals of the $\X_{n,p}$ in the polynomial ring in the
variables $x_I \in \dirwed$. Since each $\X_{n,p}$ is $G_{n,p}$-stable,
$\Xinf$ is a $G_\infty$-stable, closed subset of $(\dirwed)^*$.

\begin{ex}
For the Grassmannian $\plg_p(V):=\Gr(p,V)$, the limit $\X_\infty$ is (the
charge zero part of) {\em Sato's Grassmannian} \cite{Sato83,Segal85};
for a more algebraic treatment see \cite{Alvarez98}. It is the common
zero set in $(\dirwed)^*$ of all polynomials of the form
\[ \sum_{k=1}^\infty (-1)^k x_{I \setminus \{i_k\}} (x_{i_k} \wedge
x_{J}) \]
where $I=\{i_1<i_2<\ldots\}$ is a subset of $-\NN \cup \NN$ with $k-i_k=1$
for $k \gg 0$ (charge $1$) and $J=\{j_1<j_2<\ldots\}$ is a subset with
$k-j_k=-1$ for $k \gg 0$ (charge $-1$) and where $x_{i_k} \wedge x_J$
equals $\pm x_{J \cup \{i_k\}}$ if $i_k \not \in J$ (sign depending on the
parity of the position of $i_k$ among the $j_l$) and zero otherwise. Note
that this is, indeed, a polynomial, since $i_k \in J$ for $k \gg 0$. In
characteristic zero, these Pl\"ucker relations generate the ideal of
$\Xinf$; for positive characteristic see
\cite{Abeasis80,Bruns03}. The simplest
Pl\"ucker relation comes from $\plg_2(V_{2,2})$ and reads
\[ x_{-2,-1,3,\ldots} x_{1,2,3,\ldots} -
x_{-2,1,3,\ldots} x_{-1,2,3,\ldots} +
x_{-2,2,3,\ldots} x_{-1,1,3,\ldots},
\]
or, in the Young diagram notation:
\begin{center}
\includegraphics[width=.9\textwidth]{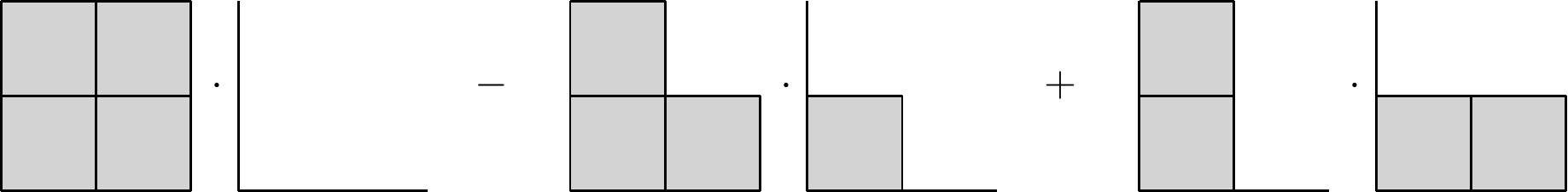}
\hfill $\diamondsuit$
\end{center}
\end{ex}

\section{Pfaffians on the dual infinite wedge}
\label{sec:pfaff}

To test whether an element $\omega$ of $\biw^2V$ has rank less than $r$,
one can use Pfaffians. We recall the definition.

\begin{de} Let $A = (a_{i,j})_{i,j=1}^{2r}$ be a skew-symmetric
matrix. Then the \emph{Pfaffian} of $A$ is defined as
$\Pf(A) = \frac{1}{2^r r!}\sum_{\sigma \in \mathrm{Sym}_{2r}}
\mathrm{sign}(\sigma)\prod_{i=1}^ra_{\sigma(2i-1)\sigma(2i)}$.  Its square
is the determinant of $A$.
\end{de}

If we write the Pfaffian of $A$ as a polynomial in
$\QQ[a_{i,j}]$, then its coefficients are integers. In fact,
if $\sigma \in \mathrm{Sym}_{2r}$, then the monomial
$\prod_{i=1}^ra_{\sigma(2i-1)\sigma(2i)}$ has coefficient
$\mathrm{sign}(\sigma)$ in $\Pf(A)$.
Hence the definition of $\Pf$ makes sense over fields of
positive characteristic, as well.

For a choice of linearly independent $x_1,\ldots,x_{2r} \in V^*$ we can
form the matrix $A=(x_i \wedge x_j)_{ij}$ of linear functions on $\Wedge^2
V$, and its Pfaffian $\Pf(A)$ is a degree-$r$ polynomial function on
$\Wedge^2(A)$. A polynomial obtained like this is called an order-$r$
sub-Pfaffian on $\Wedge^2 V$. The square of an order-$r$ sub-Pfaffian
is an order-$2r$ sub-determinant on $\Wedge^2 V$, and $\omega$ has
rank less than $r$ if and only if all order-$r$ sub-Pfaffians vanish
on it. In other words, the variety $Y^r(V)$ is the common zero set of
all $(r+1)$-th sub-Pfaffians on $\Wedge^2 V$. Note that the $(r+1)$-th
sub-Pfaffians form a single $\GL(V)$-orbit.

Returning to the $V_{n,2}$ from the definition of the infinite wedge,
observe that $\biw^2 V_{n,2}^*$ has coordinates $x_{i,j}:=x_i\wedge
x_j=-x_{i,j}$ with $i,j \in \{-n,\ldots,-2,-1,1,2\}$. We take $n=2r$
and define
\[ \Pf_{r+1} := \Pf((x_{i,j})_{i,j \in
\{-2r,\ldots,-2,-1,1,2\}}). \]
This is a polynomial function on $\biw^2 V_{n,2}^*$ and hence, by
Diagram~\eqref{eq:dirwedge}, on the dual infinite wedge. These specific
Pfaffians satisfy the following recursion that will be exploited in
Section~\ref{sec:noeth}.

\begin{lm} \label{lm:pfrec}
Assume that $r+1 \geq 2$. Among the variables $x_I$ appearing in the
Pfaffian $\Pf_{r+1}$, there is a unique $\preceq$-minimal one, namely,
$x_{-2r,-2r+1}=x_{-2r,-2r+1,3,4,\ldots}$. Moreover, we have the recursion
\[
\Pf_{r+1}=x_{-2r,-2r+1} \cdot \Pf_r \ + \ Q_{r+1}
\]
where $Q_{r+1}$ is a polynomial of degree $r+1$ in variables $\succ
x_{-2r,-2r+1,3,4,\ldots}$.
\end{lm}

\begin{proof}
The first statement is obvious, since all variables are of the form $x_i
\wedge x_j$ with $-2r \leq i < j \leq 2$ and therefore $-2r \leq i$ and
$-2r+1 \leq j$. For the second statement, note that any monomial occurring in $\Pf_{r+1}$ containing $x_{-2r,-2r+1}$ is of the form $x_{-2r,-2r+1}M_r$ with $M_r$ a monomial occurring in $\Pf_r$, and the coefficient of $M_r$ in $\Pf_r$ is the coefficient of $x_{-2r,-2r+1}M_r$ in $\Pf_{r+1}$ (namely, it is the sign of the permutation used to form $M_r$). This shows that the coefficient of $x_{-2r,-2r+1}$ in $\Pf_{r+1}$ is, indeed, $\Pf_r$.
\end{proof}

Dually, the pullback of $\Pf_{r+1}$ under a Hodge dual $\biw^{2r}
V_{2r,2} \to \biw^2 V_{2r,2}^*$ is the equation for the hypersurface
$Y^{r,\star}(V_{2r,2})$. Pulling back further along the exterior power
of an isomorphism $V_{2,2r}^* \to V_{2r,2}$, we find the dual Pfaffian
$\Pf^{\star}_{r+1}$, which is the polynomial function on $\biw^{2r}
V_{2,2r}^*$ whose vanishing characterises elements that are not of full
rank. Again, we can regard $\Pf^\star_{r+1}$ as a polynomial on the
dual infinite wedge. If we choose the scaling correctly, then we have
the following analogue of the previous lemma.

\begin{lm} \label{lm:pfrecstar}
Among the variables $x_I$ appearing in the Pfaffian
$\Pf^\star_{r+1}$, there is a unique $\preceq$-minimal one, namely,
$x_{-2,-1,1,2,\ldots,2r-2}=x_{-2,-1,1,2,\ldots,2r-2,2r+1,2r+2,\ldots}$. Moreover, we have the recursion
\[
\Pf_{r+1}^\star=x_{-2,-1,1,2,\ldots,2r-2} \cdot \Pf^\star_r \ + \
Q^\star_{r+1}
\]
where $Q^\star_{r+1}$ is a polynomial in
variables $\succ x_{-2,-1,1,2,\ldots,2r-2,2r+1,2r+2,\ldots}$.
\end{lm}

For later use, we observe that the Young diagram of the
smallest variable in $\Pf_{r+1}$ is a rectangle of width $2$ and
height $2r$, while the Young diagram of the smallest variable
in $\Pf^\star_{r+1}$ is a rectangle of height $2$ and width
$2r$:
\begin{center}
\includegraphics[scale=.7]{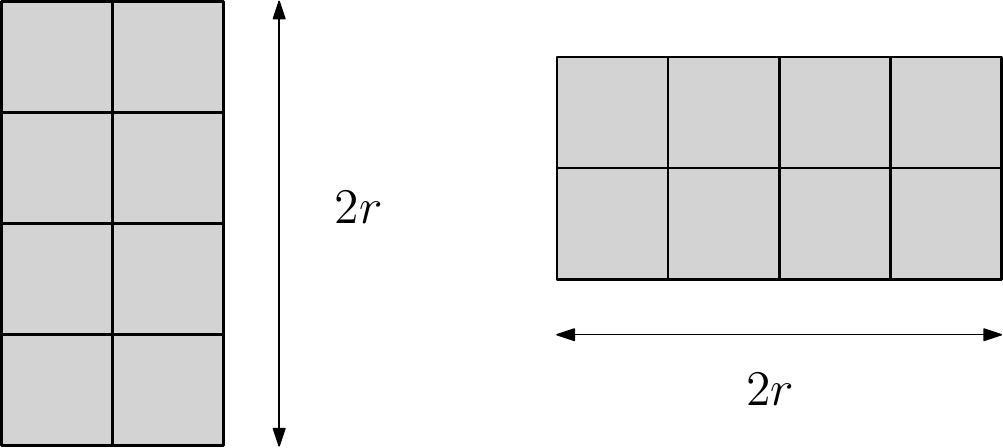}
\end{center}
All other variables have Young diagrams strictly contained in these
rectangles.

\begin{ex}
For $r=1$ we have
\[ \Pf_1 = x_{1,2} = \Pf^\star_1. \]
For $r=2$ we have
\[ \Pf_2 = x_{-2,-1}x_{1,2}-x_{-2,1}x_{-1,2} +
x_{-2,2}x_{-1,1}=\Pf^{\star}_2. \]
However, for $r=3$ we have
\begin{align*}
\Pf_3 = &x_{-4,-3}x_{-2,-1}x_{1,2} - x_{-4,-3}x_{-2,1}x_{-1,2}
+ \ldots + x_{-4,2}x_{-3,1}x_{-2,-1} \text{ and}\\
\Pf^{\star}_3 = &x_{-2,-1,1,2}x_{-2,-1,3,4}x_{1,2,3,4}
- x_{-2,-1,1,2}x_{-2,1,3,4}x_{-1,2,3,4} \\
& + \ldots + x_{-1,1,2,3}x_{-2,1,2,4}x_{-2,-1,3,4}.
\end{align*}
These polynomials are essentially different even when both
are viewed as polynomials on $\biw^4 V^*_{4,4}$, which is
the smallest $\biw^pV_{n,p}^*$ on which both are defined.
\hfill $\diamondsuit$
\end{ex}

\section{Equivariant Noetherianity of Pfaffian varieties}\label{sec:noeth}

This section contains the heart of our proof of Theorems~\ref{thm:main1}
and~\ref{thm:main2}. It deals with the following closed subsets of the
dual infinite wedge.

\begin{de}
For $r,s \in \ZZz$, we define $\Yinf^{r,s}$ as
\[ \Yinf^{r,s} := \{\omega \in (\dirwed)^* \mid
\forall g \in \Ginf: \Pf_{r+1} (g\omega)=\Pf^\star_{s+1} (g\omega)=0 \}. \]
We call $\Yinf^{r,s}$ a {\em Pfaffian variety}.
\end{de}

By construction, $\Yinf^{r,s}$ is a closed, $\Ginf$-stable subset of
the dual infinite wedge. The main result of this section is as follows.

\begin{thm}\label{thm:Ynoeth}
For all $r,s\in \ZZ_{\geq 0}$, the variety $\Yinf^{r,s}$ is
$\Ginf$-Noetherian. In other words, every $\Ginf$-stable closed subset
of $\Yinf^{r,s}$ is cut out by finitely many $\Ginf$-orbits of polynomial
equations.
\end{thm}

We will need the following lemma on the complement
of Pfaffian varieties.

\begin{lem}\label{lem:oneGorbit}
Let $\omega \in (\dirwed)^*$ and suppose that there exist $g_1,g_2 \in
\Ginf$ such that $\Pf_r(g_1\omega) \neq 0$ and $\Pf^{\star}_s(g_2\omega)
\neq 0$. Then there exists a $g \in \Ginf$ such that both $\Pf_r(g\omega) \neq 0$
and $\Pf^{\star}_s(g\omega) \neq 0$.
\end{lem}

\begin{proof}
Consider $g = \lambda g_1+\mu g_2$, where $\lambda,\mu
\in K$. Expand $\Pf_r(g\omega)$ as a formal polynomial in
$\lambda,\mu$. Observe that the coefficient at $\lambda^r$ is
$\Pf_r(g_1\omega) \neq 0$. Similarly, observe that the coefficient at
$\mu^s$ of $\Pf^{\star}_s(g\omega)$ is $\Pf^{\star}_s(g_2\omega) \neq
0$. So
the formal polynomials obtained are both non-zero, and hence the set
\[ \{(\lambda,\mu) \in K^2: g \not\in \Ginf \vee \Pf_r(g\omega) = 0 \vee
\Pf^{\star}_s(g\omega) = 0\}\]
is a proper Zariski-closed subset of $K^2$ (using the fact that $K$ is
infinite). So there exist $\lambda,\mu \in K$ such that  $g \in \Ginf$
and $\Pf_r(g\omega) \neq 0$ and $\Pf^{\star}_s(g\omega) \neq 0$, and $g
\in \Ginf$.
\end{proof}

To prove Theorem~\ref{thm:Ynoeth} we proceed by induction.  First, if
either $r=0$ or $s=0$, then the Pfaffian $\Pf_1=x_{1,2}=\Pf^\star_1$
vanishes on $\Yinf^{r,s}$. But then so do all polynomials in the
$\Ginf$-orbit of $x_{1,2}=x_{1,2,3,\ldots}$, which contains all
$x_I$. Hence then $Y^{r,s}$ consists of the single point $0$ and is
certainly equivariantly Noetherian. In the induction step, we may
therefore assume that $r,s \geq 1$. We then write
\[ \Yinf^{r,s} = \Yinf^{r-1,s} \cup \Yinf^{r,s-1} \cup Z', \]
where $Z'$ is the subset of $\omega \in \Yinf^{r,s}$ for which there
exist $g_1,g_2 \in \Ginf$ such that $\Pf_r(g_1 \omega), \Pf^\star_s(g_2
\omega)$ are both non-zero. By induction we know that the first two
terms are $\Ginf$-Noetherian, so it suffices to prove that $Z'$ is.
By the previous lemma, we have $Z'=G_\infty Z$ where
\[ Z:=\{\omega \in \Yinf^{r,s} \mid \Pf_r(\omega) \neq 0
\text{ and } \Pf^\star_s(\omega) \neq 0\}. \]
We now set out to prove that $Z$ is equivariantly Noetherian
under a suitable subgroup $H$ of $\Ginf$. To define this group, let
$G_{-\infty,-2r+1}$ denote the subgroup of $\Ginf$ of all maps that fix
all $x_i \in V_\infty$ with $i \geq -2r+2$. By Lemma~\ref{lm:pfrec}
with $r$ replaced by $r-1$, this group fixes $\Pf_r$ (and also $\Pf^\star_s$ by Lemma~\ref{lm:pfrecstar}). Similarly, let
$G_{2s-1,\infty}$ denote the group of all matrices that fix all $x_i$
with $i \leq 2s-2$.  By Lemma~\ref{lm:pfrecstar} with $r$ replaced by
$s-1$, each variable $x_I$ in $\Pf^\star_s$ has $\{2s-1,2s,2s+1,\ldots\}
\subseteq I$, so that an element $g \in G_{2s-1,\infty}$ scales $x_I$
by $\det(g)$ and hence $\Pf^\star_s$ by $\det(g)^s$ (and scales $\Pf_r$ by $\det(g)^r$). We conclude that
the open subset $Z$ of $\Yinf^{r,s}$ is stable under the group
\[ H:=G_{-\infty,-2r+1} \times G_{2s-1,\infty} \subseteq \Ginf. \]
To prove that $Z$ is $H$-Noetherian, we will embed it $H$-equivariantly
into a space of the type in Theorem~\ref{thm:main3}. To do so, we will
use the equations of $\Yinf^{r,s}$ to show that a point $\omega \in Z$ is in
fact determined uniquely by a subset of its coordinates
$x_I(\omega)$. These are the coordinates with $I$ as in the
following definition.

\begin{de}
Let $I \subseteq -\NN \cup \NN$ be a set of charge $0$. We call $I$
(or $x_I$) \emph{good} if both $I \cap \ZZ_{\leq -2r+1}$ and $I^c\cap
\ZZ_{>2s-1}$ have cardinality at most $1$.
\end{de}

A good $I$ corresponds to a lattice path that goes east at
most once to the north of the diagonal strip corresponding to $-2r+2$
and south at most once to the east of the diagonal strip
corresponding to $2s-2$, see Figure~\ref{fig:goodI}.

\begin{figure}
\includegraphics[width=.9\textwidth]{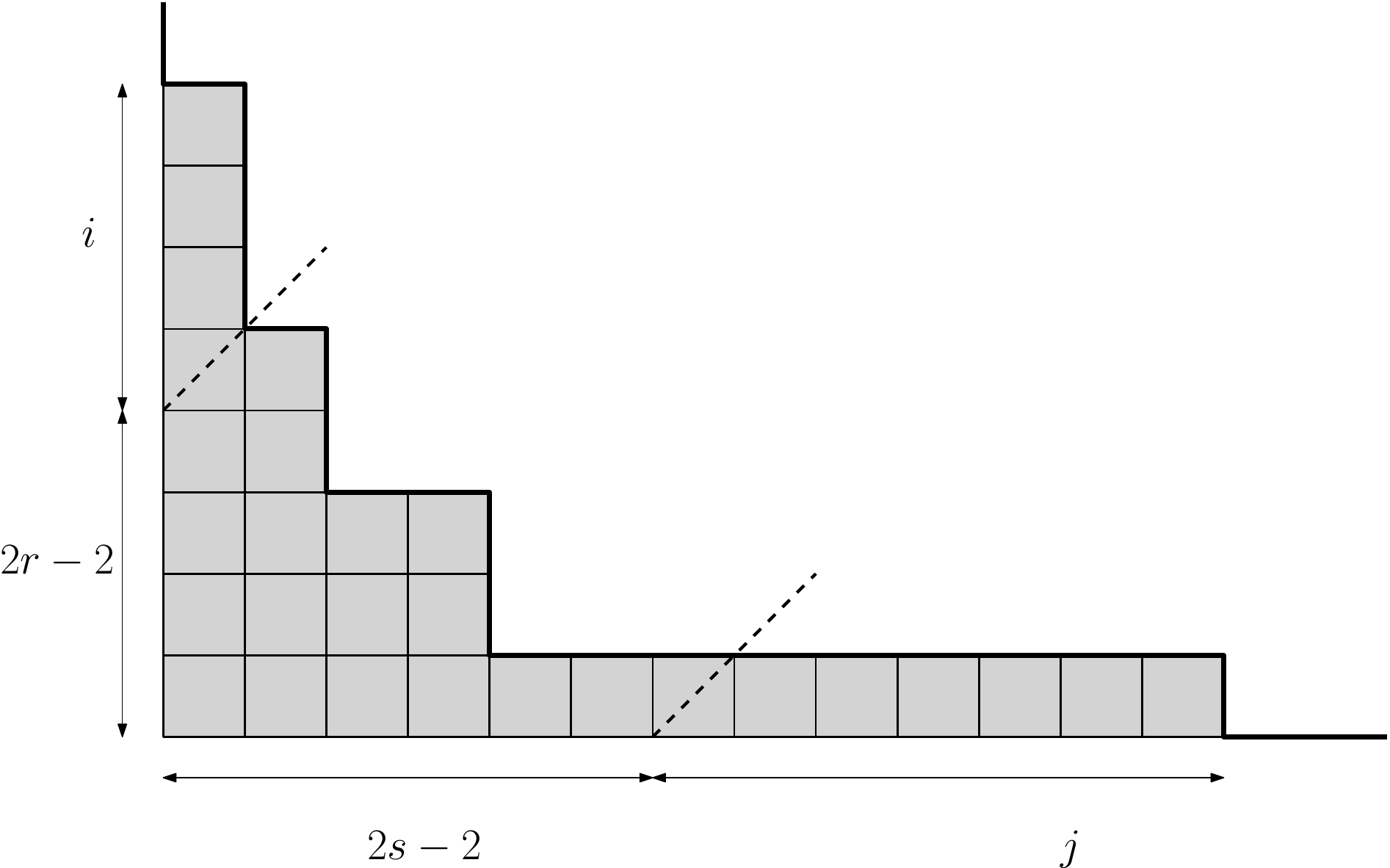}
\caption{The lattice path corresponding to a good $I$.}
\label{fig:goodI}
\end{figure}

We let $(\dirwed)_g$ be the subspace of $\dirwed$ spanned by the good
coordinates $x_I$, and let $(\dirwed)_g^*$ be its dual. Observe that $H$
acts on these spaces, and that the natural projection $(\dirwed)^* \to
(\dirwed)^*_g$ is $H$-equivariant.

\begin{lem}\label{lem:goodNoeth}
The topological space $(\dirwed)^*_g$ with the Zariski topology is
$H$-Noetherian.
\end{lem}

\begin{proof}
A coordinate $x_I$ on $(\dirwed)^*_g$ can be one of four possible types,
depending on $|I\cap \ZZ_{<-2r+2}| \in \{0,1\}$ and $|I^c\cap\ZZ_{>2s-2}|
\in \{0,1\}$. The coordinates with $|I\cap \ZZ_{<-2r+2}| =
|I^c\cap\ZZ_{>2s-2}| = 0$ form a finite set, say of size $d$. The
coordinates with $|I\cap \ZZ_{<-2r+2}| = 1$ and $|I^c\cap\ZZ_{>2s-2}| = 0$
can be organized in finitely many, say $n$, columns with row index equal
to the unique element of $I \cap \ZZ_{<-2r+2}$.  These columns are acted
upon by $G_{-\infty,-2r+1}$.  Similarly, the coordinates with $|I\cap
\ZZ_{<-2r+2}| = 0$ and $|I^c\cap\ZZ_{>2s-2}| = 1$ can be organized in
finitely many, say $m$, rows, acted upon by $G_{2s-1,\infty}$. Finally,
the coordinates with $|I\cap \ZZ_{<-2r+2}| = |I^c\cap\ZZ_{>2s-2}| =
1$ can be organized in finitely many, say $p$, infinite-by-infinite
matrices, on which $G_{-\infty,-2r+1}$ acts by row operations, and on
which $G_{2s-1,\infty}$ acts by column operations. The row and column
index of the good $x_I$ in Figure~\ref{fig:goodI}, for instance, equals
$-2r+2-i$ and $2s-2+j$.

Thus, after relabelling the column and row indices to take values in
$\NN$ (so replacing $-2r+2-i$ by $i$ and $2s-2+j$ by $j$), an element
of $(\dirwed)^*_g$ can be seen as a tuple of a vector in $K^d$, an
element of $K^{\NN \times n}$, and element of $K^{m \times \NN}$, and an
element of $(K^{\NN \times \NN})^p$, and the action of $H$ corresponds
to the diagonal action of $\GL_\NN \times \GL_\NN$ on this space.
Now Theorem~\ref{thm:main3} implies the lemma.
\end{proof}

Before we continue with the proof that $Z$ is $H$-Noetherian, we recall
that the coordinates $x_I$ are partially ordered by $\preceq$, which
corresponds to opposite containment of Young diagrams. This order is
compatible with the ``upper triangular'' derivations $\partial_{k,l},\
k < l$ from Example~\ref{ex:derivations} in the following sense:
first, if $\partial_{k,l} x_I$ is non-zero, then it equals $\pm X_J$ with
$J \prec I$. Second, if also $\partial_{k,l} x_K=\pm X_L$ is non-zero
and if $I \prec K$, then $J \prec L$.

We now consider the projection $Z \to (\dirwed)^*_g$ that takes a point
$\omega$ and forgets all its coordinates except for the good ones. We
claim that this map is injective, and in fact a closed embedding into
the open subset of $(\dirwed)^*_g$ where both $\Pf_r$ and $\Pf^\star_s$
are non-zero. For this it suffices to show that, on $Z$, each coordinate
$x_I$ can be expressed as a rational function in the good coordinates,
whose denominator only has factors $\Pf_r$ and $\Pf^\star_s$. If $I$
is good, then $x_I$ itself is such an expression. Now we proceed by
induction relative to the partial order $\preceq$. So let $I$ be not good,
and assume that for all $J \succ I$ such a rational expression exists for
$x_{J}$. Since $I$ is not good, one of the following two cases applies.

First, suppose $|I \cap \ZZ_{\leq -2r+1}| > 1$. Then $I \preceq
\{-2r,-2r+1,3,4,\ldots\}$. On $\Yinf^{r,s}$ we have
\[ 0=\Pf_{r+1}=x_{-2r,-2r+1,3,4,\ldots}\Pf_r + Q_{r+1} \]
by Lemma~\ref{lm:pfrec}, where both $\Pf_r$ and $Q_{r+1}$ contain only
variables $x_J$ with $J \succ \{-2r,-2r+1,3,4,\ldots\}$.  Write $I=\{i_1 <
i_2 < \ldots\}$ and let $k \geq 2$ be the maximal index for which $i_k <
k$. Consider the differential operator
\[ D:=\partial_{i_k,k} \circ \partial_{i_{k-1},k-1} \circ \cdots \circ
\partial_{i_3,3} \circ \partial_{i_2,-2r+1} \circ \partial_{i_1,-2r} \]
which is chosen such that $D x_{-2r,-2r+1,3,4,\ldots}=x_I$. We stress
the order: first $\partial_{i_1,-2r}$ has the effect of replacing $-2r$
by $i_1$, then $-2r+1$ is replaced by $i_2$, etc. Since
$\Yinf^{r,s}$ is $\Ginf$-stable, the ideal of polynomials vanishing on
it is stable under $D$. Applying $D$ to the equation above, and using
the Leibniz rule, we find that
\[ 0 = x_I \Pf_r + P + D Q_{r+1} \]
holds on $\Yinf^{r,s}$, where $P$ is obtained from the product
$x_{-2r,-2r+1,3,4,\ldots} \cdot \Pf_r$ by letting at least one of
the factors of $D$ act on $\Pf_r$ and the remaining factors act on
$x_{-2r,-2r+1,3,4,\ldots}$. By the discussion above, the variables
appearing in $P$ and in $D Q_{r+1}$ are all strictly greater than $x_I$,
so for those variables a rational expression exists as
desired by the induction hypothesis. But then also for
\[ x_I = (-P-D Q_{r+1})/\Pf_r \]
such an expression exists.

Second, otherwise we have $|I^c \cap \ZZ_{\geq 2s-1}|>1$.
Then we have
\[ I \preceq \{-2,-1,1,2,\ldots,2s-2,2s+1,2s+2,\ldots\}. \]
On $\Yinf^{r,s}$ we have
\[
0= \Pf_{s+1}^\star=x_{-2,-1,1,2,\ldots,2s-2} \cdot \Pf^\star_s \ + \
Q^\star_{s+1}
\]
where all variables in $\Pf^\star_s$ and $Q^\star_{s+1}$ are strictly
larger than $x_{-2,-1,1,2,\ldots,2s-2}.$
Again, write $I=\{i_1<i_2<\ldots\}$, let $k \geq 2s$ be
maximal with $i_k \neq k$ and apply the differential operator
\[ D= \partial_{i_k,k} \circ \cdots \circ
\partial_{i_{2s+1},2s+1} \circ \partial_{i_{2s},2s-2}
\cdots \circ \partial_{i_3, 1} \circ \partial_{i_2,-1} \circ
\partial_{i_1,-2} \]
to the equation above to find an expression for $x_I$.

We conclude that the topological space $Z$ is isomorphic to an $H$-stable
locally closed subset of $(\dirwed)^*_g$ with the induced topology. Since
the latter space is $H$-Noetherian, so is $Z$. A basic observation
on equivariant Noetherianity is that when the $H$-Noetherian space $Z$
is smeared out by the larger group $\Ginf \supseteq H$, then the
resulting topological space is $\Ginf Z \subseteq (\dirwed)^*$ is
$\Ginf$-Noetherian (see \cite[Lemma 5.4]{Draisma11d}). Moreover, the union of finitely
many $\Ginf$-Noetherian spaces is $\Ginf$-Noetherian, hence in particular
so is
\[ \Yinf^{r,s} = \Yinf^{r-1,s} \cup \Yinf^{r,s-1} \cup \Ginf Z. \]
This concludes the proof of Theorem~\ref{thm:Ynoeth}
given Theorem~\ref{thm:main3}; the latter will be proved in
Section~\ref{sec:mxtuples}. A direct consequence of that theorem is
the following.

\begin{cor} \label{cor:pvlimit}
Let $\plg$ be a bounded \pv{}. Then its limit $\Xinf$ is defined in
$(\dirwed)^*$ by the $\Ginf$-orbits of finitely many polynomial equations.
\end{cor}

\begin{proof}
The limit $\Xinf$ is a closed, $\Ginf$-stable subset of $\Yinf^{r,r}$,
where $r$ is the rank of the bounded \pv{}. Since $\Yinf^{r,r}$ is
$\Ginf$-Noetherian, $\Xinf$ is defined within $\Yinf^{r,r}$ by the
$\Ginf$-orbits of finitely many equations. Adding to these the single
$\Ginf \Pf_r=\Pf^\star_r$, the corollary follows.
\end{proof}

\section{Back to finite-dimensional instances}\label{sec:backfin}

We have now proved a fundamental result, Corollary~\ref{cor:pvlimit},
on the limit of a bounded \pv{} $\plg$. This is a projective limit
of all instances of the \pv{}, hence projects to all of them. To draw
conclusions for these instances themselves, however, we need to be able
to lift them back into the limit.  For this we make use
of Diagram~\eqref{eq:dirwedgestar}. This diagram allows us to extend
a single $\omega_{n_1,p_1} \in \Wedge^{p_1} V^*_{n_1,p_1}$ to a point
$\omega=(\omega_{n,p})_{n,p \in \ZZz}$ in the dual infinite wedge. By
the right-most diagram in~\eqref{eq:diagX}, this point $\omega$ lies in
$\Xinf$ if and only if $\omega_{n_1,p_1}$ lies in $X_{n_1,p_1}$. We can
now easily prove the main theorem.

\begin{proof}[Proof of the main theorem.]
By Corollary~\ref{cor:pvlimit} there exist $n_0,p_0$ such that the
$\Ginf$-orbits of the equations of $X_{n_0,p_0}$ define $\Xinf$. Thus
for arbitrary $n,p$ we can find polynomials
$f_1,\ldots,f_N$ in the ideal of $X_{n_0,p_0}$ and group elements
$g_1,\ldots,g_N \in \Ginf$ such that the $g_i f_i$, restricted to
$\Wedge^p V^*_{n,p}$ via Diagram~\ref{eq:dirwedgestar}, define $X_{n,p}$
set-theoretically. Each of these equations $g_i f_i$ arises as a pullback
of $f_i$ under a sequence of linear maps as in the definition of a
\pv{}.
\end{proof}

This proof is slightly unsatisfactory in that we seem to have no control
over the elements $g_i$. {\em A priori}, they may have to be chosen from
$G_{n',p'}$ with $n',p'$ much larger than the relevant $n,n_0,p,p_0$. Our
next goal is to show that this is not the case. After the $g_i$ are
under control, also Theorem~\ref{thm:main2} follows.

From now on, we write $\Projd_{n_0,p_0} \omega$ for the image in
$\Wedge^{p_0} V_{n_0,p_0}^*$ of a point $\omega$ in the dual infinite
wedge. We will use the same notation when $\omega$ lies in some finite
$\Wedge^{p} V_{n,p}^*$. Then it is understood that $\omega$ is first
lifted to the dual infinite wedge and then projected.

\begin{lm}~\label{lem:onestep}
Let $\plg$ be a \pv{}, $n,p \in \ZZz$, and $\omega \in \biw^p V_{n,p}^*$.
Let $n_0,p_0,n',p' \in \ZZz$ with $n' \geq \max\{n,n_0\}$
and $p' \geq \max\{p,p_0\}$ and
suppose that there exists a $g \in G_{n',p'}$ such that $\Projd_{n_0,p_0}
g(\omega) \not\in \X_{n_0,p_0}$. Then the following hold.
\begin{enumerate}
\item If $p' > p$ and $p' > p_0$, then $\exists g' \in
G_{n',p'-1}: \Projd_{n_0,p_0}g'(\omega) \not\in \X_{n_0,p_0}$.
\item If $p' > p$ and $p'= p_0$, then $\exists g' \in
G_{n',p'-1}: \Projd_{n_0,p_0-1}g(\omega) \not\in \X_{n_0,p_0-1}$.
\item If $n' > n$ and $n' > n_0$, then $\exists g' \in
G_{n'-1,p'}: \Projd_{n_0,p_0}g(\omega) \not\in \X_{n_0,p_0}$.
\item If $n' > n$ and $n' = n_0$, then $\exists g' \in
G_{n'-1,p'}: \Projd_{n_0-1,p_0}g'(\omega) \not\in \X_{n_0-1,p_0}$.
\end{enumerate}
\end{lm}

\begin{proof}
The first part and the third part of the lemma are dual to each other, and
so are the second part and the fourth part.
%Namely,
%we have $\Projd_{n_0,p_0}g (\omega) \not\in \X_{n_0,p_0}$
%if and only if $\Proj_{p_0,n_0}\star g\Incld_{p',n'}(\star \omega)
%\not\in \star\X_{n_0,p_0}$. Here, $\star g = \ldots$.
Therefore, it suffices to prove only the first two parts.
Moreover, the condition $\pi_{n_0,p_0}g(\omega) \not \in
X_{n_0,p_0}$ holds for $g$ in a nonempty and open, hence
dense, subset of
$G_{n',p'}$, so we may assume that $g$ from the statement of the lemma is
sufficiently general.  Recall that $V_{n',p'}^* = \langle e_i \rangle_{-n'
\leq i \leq p', i \neq 0}$, with corresponding coordinates $x_i$.

Suppose that $p' > p$. We may assume that $x_{p'}(ge_{p'}) \neq 0$. We
define the linear map $g'$ on $V_{n',p'-1}^*$ by
\[ g'v = gv-\frac{x_{p'}(gv)}{x_{p'}(ge_{p'})}ge_{p'},\ v
\in V_{n',p'-1}^*.\]
In other words, $g'$ equals the composition of
$g|_{V^*_{n',p'-1}}:V^*_{n',p'-1} \to V^*_{n',p'}$ and the projection
$V^*_{n',p'} \to V^*_{n',p'-1}$ along $ge_{p'}$.  We view $g'$
as an element of $G_{n',p'}$ by inclusion, i.e., fixing $e_{p'}$.
In $\biw^{p'}V_{n',p'}^*$ we compute
\[ g'\pi_{n',p'-1}(\omega) \wedge ge_{p'} =
g \Projd_{n',p'-1}(\omega) \wedge g e_{p'} =
g(\Projd_{n',p'-1}(\omega) \wedge e_{p'}) =
g(\Projd_{n',p'}(\omega))
. \]
Here the first equality follows from basic properties of alternating
tensors and the last equality follows from $p'>p$, which means that to go
from $\omega$ to $\pi_{n',p'}\omega$ one tensors with $p'-p > 0$ factors
on the right, and then follows the inclusion $\Wedge^{p'}V_{n,p'}^*
\to \Wedge^{p'}V_{n',p'}^*$. Contracting both sides with $x_{p'}$ yields
\begin{equation} \label{eq:lowerp}
x_{p'}(g e_{p'}) \cdot g' \pi_{n',p'-1}(\omega) =
\pi_{n',p'-1}(g \omega). \end{equation}
Now if $p' > p_0$, then we can further down and find
\[ \Projd_{n_0,p_0}g'(\omega) = \frac{1}{x_{p'}(ge_{p'})}\Projd_{n_0,p_0}
g (\omega) \not\in \X_{n_0,p_0}. \]
If $p'=p_0$, then wedging the right-hand side of~\eqref{eq:lowerp}
with $g e_{p_0}$ one obtains $\pi_{n',p_0}(g \omega)$, which projects
to $\pi_{n_0,p_0}(g \omega) \not \in \X_{n_0,p_0}$. This element is also
obtained from the left-hand side by applying $\pi_{n_0,p_0-1}$ and then
wedging with the projection of $g e_{p_0}$ to $V_{n_0,p_0}^*$. Hence the element $\pi_{n_0,p_0-1} g' \omega$
does not lie in $\X_{n_0,p_0-1}$.
\end{proof}

\begin{cor} \label{cor:boundedg}
Let $\plg$ be a bounded \pv{}. Then there exist $n_0,p_0
\in \ZZz$ such that for all $n,p \in \ZZz$, and all $\omega \in
\biw^pV_{n,p}^*$, the following are equivalent:
\begin{enumerate}
\item $\omega \not\in X_{n,p}$.
\item There exists $g \in G_{n,p}$ such that $\Projd_{\min(n,n_0),\min(p,p_0)}(g\omega) \not\in X_{\min(n,n_0),\min(p,p_0)}$.
\end{enumerate}
\end{cor}

\begin{proof}
By Corollary~\ref{cor:pvlimit}, there exist $n_0,p_0$ such that for all
$n,p$ and $\omega \in \biw^p V_{n,p}^*$ we have $\omega \not \in X_{n,p}$
if and only if $\exists g \in \Ginf: \Projd_{n_0,p_0} (g \omega) \not
\in X_{n_0,p_0}$. Now apply Lemma~\ref{lem:onestep} repeatedly to get $g$
down to $G_{n,p}$.
\end{proof}

We conclude this section with the proof of
Theorem~\ref{thm:main2}.

\begin{proof}[Proof of Theorem~\ref{thm:main2}.]
Let $n_0,p_0$ be as in the previous corollary, and let $f_1,\ldots,f_N$
be defining equations for $\X_{n_0,p_0}$.

Let $(d,p,\omega \in \Wedge^p K^d)$ be the input to the algorithm. We
first give a {\em randomised} algorithm for testing wether $\omega \in
\plg_p(K^d)$. First, if $p>d$, then $\omega=0$ and the output is {\em
yes} if $\plg$ is not the empty \pv{} and {\em no} if it is. Otherwise,
set $n:=d-p$, pick a random linear isomorphism $g: K^d \to V_{n,p}^*$,
and return the answer to the question whether
$f_k(\pi_{n_0,p_0}\Wedge^p g(\omega))=0$ for all $k=1,\ldots,N$.

If $\omega$ lies in $\plg_p(K^d)$, then the output will always
be {\em yes}. If $\omega$ does not lie in $\plg_p(K^d)$, then by
Corollary~\ref{cor:boundedg} an open and dense set of choices for $g$
will yield the correct output {\em no}. Clearly, the number of arithmetic
operations over $K$ is polynomially bounded. Moreover, since the $f_k$
are fixed polynomials, no super-polynomial coefficient blow-up can happen
if one works over $\QQ$ or a more general number field.

To make this algorithm deterministic, one can take the matrix entries
$g_{ij}$ of $g$ to be variables rather than elements of $K$, and output
{\em yes} if all $f_k(\pi_{n_0,p_0}\Wedge^p g(\omega))$ are zero {\em
as polynomials in those variables}. Now the arithmetic operations
take place in the polynomial ring $K[g_{ij}]$, but (again since the $f_i$
are fixed) they still reduce to polynomially many operations
over $K$, and to an algorithm of polynomial bit-complexity
over $\QQ$ or number fields.
\end{proof}

\section{Noetherianity of matrix tuples}
\label{sec:mxtuples}

We recall the statement of Theorem~\ref{thm:main3}: for all ${\p},n,m,d\in \ZZz$, the space
\[ \Tot_{{\p},n,m,d} = \Matp\times \Mat_{\NN,n} \times
\Mat_{m,\NN} \times \Fin\]
is $\GLprod$-Noetherian with respect to the Zariski topology.

In the case where $p=0$, a much stronger statement is known to hold: the
coordinate ring of this variety is $\Sym(\NN) \times \Sym(\NN)$-Noetherian
\cite{Hillar09}. But this fails for $p>0$ \cite[Example 3.8]{Hillar09}, and we will not need this
result in our proof. The entire section will be devoted to the proof of
the theorem.  We order $\ZZz^4$ lexicographically, and we apply induction
on $({\p},n,m,d)$. From here on, we assume that $\Tot_{{\p}',n',m',d'}$
is $\GLprod$-Noetherian whenever $({\p}',n',m',d')$ is lexicographically
smaller than $({\p},n,m,d)$. We abbreviate $\Tot = \Tot_{{\p},n,m,d}$.
For $\x \in \Tot$ we write $\xma$, $\xcol$, and $\xrow$ for the
projections of $\x$ in $\Matp$, $\Colp$, $\Rowp$, respectively.

A key step in our proof will be a version of the the following dichotomy:
a $\GLprod$-stable closed subset of $\Matp$ is either equal to $\Matp$
or else consists of matrix tuples of {\em bounded rank}, in the sense
of the following definition.

%\subsection{Induction part}
\begin{de}
For a tuple $\ma = (\ma_1,\ldots,\ma_{\p})$ of matrices of the same size
(infinite or not), we define the {\em rank} as
\[ \rk(\ma) =
\min\{\rk(c_1\ma_1+\ldots+c_{\p}\ma_{\p}) \mid (c_1:\cdots:c_p)
\in
\PP^{p-1} \} \in \ZZz \cup \{\infty\}.\]
\end{de}

\noindent Before establishing the dichotomy, we settle the bounded-rank case
by induction.

\begin{lm}\label{lm:finrk}
Fix $r \in \ZZz$. Then $\{x \in \Tot \mid \rk \xma \leq r\}$ is
$\GLprod$-Noetherian.
\end{lm}

\begin{proof}
Consider the morphism $\varphi:\Tot_{p-1,n+r,m+r,d+p^2} \to
\Tot_{p,n,m,d}=\Tot$ defined as follows. Let $(M_1,\ldots,M_{p-1}) \in
\Mat_{\NN,\NN}^{p-1}$, let $C_1 \in \Mat_{\NN,m}, C_2 \in
\Mat_{\NN,r}$, let $R_1 \in \Mat_{m,\NN},R_2 \in
\Mat_{r,\NN}$, and let $t \in K^d$ and $(\alpha_{ij}) \in
K^{p \times p}$. Then
\[ \varphi((M_1,\ldots,M_{p-1}),(C_1,C_2),(R_1,R_2),(t,\alpha))
:= (\xma,C_1,R_1,t) \]
where the $i$-th matrix in $\xma$ equals
\[ \sum_{j=1}^{p-1} \alpha_{ij} M_j + \alpha_{ip} C_2 \cdot
R_2.
\]
This map is $\GLprod$-equivariant, continuous, and its image equals
the set in the statement of the lemma. Since $(p-1,n+r,m+r,d+p^2)$
is lexicographically smaller than $(p,n,m,d)$, the left-hand space
is equivariantly Noetherian by the induction assumption. Hence so is
its image.
\end{proof}

Similar induction arguments apply to the set of $x \in \Tot$ for which
$\xcol \in \Mat_{\NN,n}$ has rank strictly less than $n$ or $\xrow$
has rank strictly less than $m$. So we need only focus on the $x$ for
which $\xcol$ and $\xrow$ have full rank, and $\xma$ has high rank. We
start with an easy lemma in linear algebra.

%\subsection{The real work}

\begin{lm}\label{lm:matrixreduction} Let $N_1,N_2 \in \ZZz$, and let $\ma
\in \Matfin$. If $\ma$ has rank at least ${\p}$,
then there exists a $v \in K^{N_2}$ for which $\ma_1v,\ldots,
\ma_{\p}v \in K^{N_1}$ are linearly independent.
\end{lm}

\begin{proof}
Consider the variety
\[ Y := \{(v,d) \in K^{N_2}\times \PP^{p-1} \mid
(d_1\ma_1+\ldots+d_{\p}\ma_{\p})\cdot v = 0\}. \]
Given $d \in \PP^{p-1}$, the space $\{v \mid (v,(d_i)_{i=1}^{\p}) \in Y\}$
has dimension at most $N_2-p$ because $\rk(d_1\ma_1+\ldots+d_{\p}\ma_{\p})
\geq p$ by assumption.  In other words, the fibre in $Y$ above $d$ has dimension
at most $N_2-p$. But then $Y$ has dimension at most $N_2-
p + p-1 < N_2$, and hence the projection from $Y$ to $K^{N_2}$ is not
surjective. \end{proof}

\begin{cor}\label{cor:indepvectfin}
Let $N_1,N_2 \in \ZZz$, let $\ma \in \Matfin$, and let $l \in \ZZz$. If
$\ma$ has rank at least ${\p}l$, then there exists a linear
subspace $V \subseteq K^{N_2}$ of dimension $l$ such that $\ma_1V +
\ldots + \ma_{\p}V$ has dimension ${\p}l$.
\end{cor}

\begin{proof}
We apply induction. For $l = 1$ the corollary follows from
Lemma~\ref{lm:matrixreduction}. Let $l > 1$, and assume that the
corollary is true for $l-1$. Use the lemma to pick $v \in K^{N_2}$
such that $W:=\langle \ma_1v,\ldots, \ma_{\p}v \rangle$ has dimension
${\p}$. Then each $\ma_i$ induces a linear map
\[ \tilde{\ma}_i:K^{N_2}/\langle v \rangle \to K^{N_1}/W,\]
and the tuple $(\tilde{\ma}_1,\ldots,\tilde{\ma}_p)$ has rank at least
${\p}l-{\p}={\p}(l-1)$---indeed, deleting a row or column from
a matrix tuple reduces the rank by at most one, and deleting a zero row from a matrix tuple does not reduce the rank. By the induction
hypothesis we find an $(l-1)$-dimensional $V' \subseteq K^{N_2}/\langle
v \rangle$ such that $\dim (\tilde{\ma}_1 V' + \cdots + \tilde{\ma}_p
V')={\p}(l-1)$, and the preimage $V$ of $V'$ in $K^{N_2}$ has
the desired property.
\end{proof}

\begin{cor}\label{cor:indepvect} Let $\ma \in \Matp$ and let $l \in
\ZZz$. If $\ma$ has rank at least ${\p}l$, then there exists a
linear space $V \subseteq K^{(\NN)}$ of dimension $l$ such that $\ma_1V +
\ldots + \ma_{\p}V$ has dimension ${\p}l$.
\end{cor}

Here, by $K^{(\NN)}$ we mean the countable-dimensional subspace of
$K^\NN$ where all but finitely many coordinates are zero. A matrix
in $\Mat_{\NN,\NN}$ defines naturally a linear map
$K^{(\NN)} \to K^\NN$, which is referred to in the
corollary.

\begin{proof}
For $N \in \ZZz$, denote by $\pi_N$ the projection from $\Mat_{\NN,\NN}$ to
$\Mat_{N,N}$. Define the variety
\[ D_N = \{(d_1:\ldots:d_{\p}) \in \PP^{p-1} \mid
\rk(\pi_N(d_1\ma_1+\ldots+d_{\p}\ma_{\p})) < {\p}l\}. \]
Observe that $D_1 \supseteq D_2 \supseteq \ldots$ is a descending sequence
of closed subvarieties of $\PP^{p-1}$. Moreover, the intersection of
the $D_N$ is $\emptyset$ because $\ma$ has rank at least ${\p}l$. So
there exists an $N \in \ZZz$ such that
$\rk(\pi_N(\ma_1),\ldots,\pi_N(\ma_{\p})) \geq {\p}l$.
Now apply Corollary~\ref{cor:indepvectfin} to find a linear subspace $V
\subseteq K^N$ such that $\pi_N(\ma_1)V + \ldots + \pi_N(\ma_{\p})V$ has
dimension ${\p}l$. View $V$ as a subspace of $K^{(\NN)}$, and observe that
$\ma_1V + \ldots + \ma_{\p}V$ has dimension ${\p}l$, as
desired.
\end{proof}

For $N_1,N_2 \in \ZZz$, we denote $\Totfin^{N_1,N_2} :=
\Matfin \times \Colfin \times \Rowfin$. Note that we have a natural projection from $\Tot$ to $\Totfin^{N_1,N_2}$.

\begin{prop}\label{prop:highrank-denseorbit}
Let $N_1, N_2 \in \ZZz$. Then there exists an $r \in \ZZz$ such that for
any $\x \in \Tot$ with $\rk(\xma) \geq r$ and with $\xcol$ and $\xrow$
of full rank, the projection of the orbit $\GLprod \x$ to $\Totfin^{N_1,N_2}$
is dense in $\Totfin^{N_1,N_2}$.
\end{prop}

\begin{proof}
It suffices to prove the lemma for $N_1 = {\p}l$ and $N_2 = l$ with
$l \in \ZZz$. Let $r:= n+m+{\p}l$ and let $\x \in \Tot$ be as in
the statement of the proposition.  As in Corollary~\ref{cor:indepvect},
there exists an $N$ such that the projection of $\xma$ to the first $N
\times N$ coordinates has rank at least $r$. Without loss of generality
(taking $N$ larger if necessary), we assume that the projection of $\xcol$
in $\Mat_{N,n}$ and the projection of $\xrow$ in $\Mat_{m,N}$ have
full rank.  From now on, we view $\x$ as an element of $\Totfin^{N,N}$,
and only act on it with elements of $\GL_N \times \GL_N$---the first
copy on $\xma,\xcol$ by row operations, and the second copy on $\xma,\xrow$
by column operations.

Without loss of generality, we may assume that
\[ \xcol=\begin{bmatrix} 0_{N-n,n}\\ I_n \end{bmatrix}
\text{ and }
\xrow=\begin{bmatrix} 0_{m,N-m} & I_m \end{bmatrix}.
\]
Let $\ma' = \xma'$ be the projection of $\ma = \xma$ to
$(\Mat_{N-n,N-m})^{\p}$. Observe that $\ma'$ has rank at least
${\p}l$. Then by Corollary~\ref{cor:indepvectfin}, there exists
a linear subspace $V \subset K^{N-m}$ of dimension $l$ such that
$W:=\ma_1'V+\ma_2'V+\ldots+\ma_{\p}'V \subset K^{N-n}$ has dimension ${\p}l$.
View $V$ as a subspace of $K^N$. Performing column
operations on the first $N-m$ columns does not change
$\xrow$ and can bring $V$ into the span of the first $l$
standard basis elements. Performing row operations on the
first $N-m$ rows does not change $\xcol$ and can bring $W$
into the span of the first $pl$ standard basis vectors.
Performing further row operations on the first $pl$ rows we achieve that
\[
\xma=\left(
\begin{bmatrix}
I_l & *_{l,N-l} \\
0_l & *_{l,N-l} \\
\vdots & \vdots \\
0_l & *_{l,N-l} \\
*_{N-pl,l} & *_{N-pl,N-l}
\end{bmatrix},
\begin{bmatrix}
0_l & * \\
I_l & * \\
\vdots & \vdots \\
0_l & * \\
*   & *
\end{bmatrix},
\ldots,
\begin{bmatrix}
0_l & * \\
0_l & * \\
\vdots & \vdots \\
I_l & * \\
*   & *
\end{bmatrix}
\right).
\]
Subtracting suitable linear combinations of the first $pl$
rows from the last $N-pl$ rows, we can clear the $*$s below
the identity matrices:
\[
\xma=\left(
\begin{bmatrix}
I_l & *_{l,N-l} \\
0_l & *_{l,N-l} \\
\vdots & \vdots \\
0_l & *_{l,N-l} \\
0_{N-pl,l} & *_{N-pl,N-l}
\end{bmatrix},
\begin{bmatrix}
0_l & * \\
I_l & * \\
\vdots & \vdots \\
0_l & * \\
0   & *
\end{bmatrix},
\ldots,
\begin{bmatrix}
0_l & * \\
0_l & * \\
\vdots & \vdots \\
I_l & * \\
0   & *
\end{bmatrix}
\right),
\]
still with $\xrow$ and $\xcol$ as above. To see that the $\GL_N \times
\GL_N$-orbit of $x$ projects dominantly into $\Totfin^{pl,l}$, pick a
general point $x'$ in the latter space. Subtract a linear combination of
the last $N-m$ columns of $\xma$ and $\xrow$ from the first $l$ columns
to achieve that $\xrow$ becomes equal to $\xrow'$.  This messes up the
$0/I$-structure of $\xma$, but by generality of $\xcol'$ we may assume
that the $pl \times pl$-matrix obtained from the new $\xma$ by concatenating
the $pl \times l$-submatrices of the $p$ components is still invertible
(though no longer the identity matrix). Moreover, we may assume that the
same holds for $\xma'$. Then suitable row operations with the first $pl$
rows move $\xma$ into $\xma'$, and subtracting suitable multiples of these
$pl$ rows from the $N-pl$ rows below them again clears the $*'s$. Finally,
subtracting a suitable linear combination of the last $N-n$ rows from the
first $pl$ rows fixes $\xma=\xma'$ and $\xcol=\xcol'$ but moves $\xrow$
into $\xrow'$.
\end{proof}

Here is the promised dichotomy.

\begin{cor} \label{cor:dichot}
Let $Y$ be a $\GLprod$-stable closed subset of $\Matp \times \Colp
\times \Rowp$. If $Y$ contains elements $x$ with $\xrow$ and $\xcol$
of full rank and $\xma$ of arbitrarily high rank, then $Y=\Matp \times
\Colp \times \Rowp$.
\end{cor}

\begin{proof}
Let $f$ be a polynomial that vanishes identically on $Y$.  Then there
exist $N_1,N_2$ such that the matrix entries appearing in $f$ are
coordinates on $\Totfin^{N_1,N_2}$. Let $r$ be as in the proposition,
and pick an element $x \in Y$ with $\rk \xma=r$ and $\xcol,\xrow$
of full rank. Then $f$ vanishes on the projection of $\GLprod x$ in
$\Totfin^{N_1,N_2}$. Hence, by the proposition, $f$ is zero.
\end{proof}

Now we can complete the proof of the theorem.

\begin{proof}[Proof of Theorem~\ref{thm:main3}.]
Assume that we have a descending chain
\[ \Tot \supseteq Y_1 \supseteq Y_2 \supseteq \ldots \]
of $\GLprod$-stable closed subsets of $\Tot$. For $r \in \ZZz$ let $U_k^r$
denote the subset of $Y_k$ where $\xcol$ and $\xrow$ have full rank and
$\xma$ has rank at least $r$. By Lemma~\ref{lm:finrk} (and variations
concerning $\xcol$ and $\xrow$), it suffices to prove that for some
value of $r$, the chain of closures
\[ \Tot \supseteq \overline{U_1^r} \supseteq \overline{U_2^r} \supseteq \ldots \]
stabilises. Note that $U_k^r$ shrinks both when $k$ grows and when $r$
grows. Since $K^d$ is Noetherian, we may choose $k$ and $r$ such that for
all $l \geq k$ and $s \geq r$ the closure of the projection of $U_l^s$
into $K^d$ is constant, say equal to $Z$. We will then argue that
\[ \overline{U_l^r} = \Matp \times \Mat_{\NN,n} \times
\Mat_{m,\NN} \times Z \text{ for all } l \geq k. \]
To simplify notation, write $X:=U_l^r$, and let $f$ be a polynomial
function on $\Tot$ that vanishes identically on $X$.  We can write
$f = \sum_{i=1}^l f_i \otimes h_i$ with each $f_i$ defined on $\Matp
\times \Colp \times \Rowp$, and each $h_i$ defined on $K^d$, and with
the $f_i$ linearly independent. We will argue that each $h_i$ vanishes
on $Z$.

Choose $N_1,N_2$ such that all $f_i$ involve only coordinates from
$\Totfin^{N_1,N_2}$. By the proposition there exists an $s \geq r$ such
that for all $x \in \Matp \times \Colp \times \Rowp$ with $\xcol,\xrow$ of
full rank and $\rk \xma \geq s$ the orbit $\GLprod x$ projects dominantly
into $\Totfin^{N_1,N_2}$. In particular, the restrictions of the $f_i$
to $\GLprod x$ are still $K$-linearly independent. Now if such an $x$
lies in the fibre in $X$ over $z \in Z$, then we find
\[ 0=f(g(x,z))=f(gx,z)=\sum_i f_i(gx) h_i(z) \text{ for all
} g \in \Ginf,\]
and by varying the $g$ we conclude that $h_i(z)=0$ for all $i$. Finally,
the image of $U_l^s \subseteq X$ in $Z$ is dense by assumption, and
hence $h_i(z)=0$ for all $i$ and $z$.
\end{proof}

\section{Discussion}\label{sec:disc}

We have introduced the natural notion of {\em Pl\"ucker variety}, which
is a family of subvarieties of exterior powers that, like Grassmannians,
are functorial and behave well under duals. For the {\em bounded ones}
among these, we have established that they are defined by polynomial
equations of bounded degree, independently of the particular instance
of the Pl\"ucker variety. This result is new already for the first
secant variety of the Grassmannian, and for the tangential variety of
the Grassmannian.

Before turning to several open questions that result from our
work, let us explain the second part of the title of this paper. In
Section~\ref{sec:limit} we have seen that the projective limit of all
(cones over) Grassmannians is (the cone over) Sato's Grassmannian. Higher
secant varieties of any Pl\"ucker variety are again Pl\"ucker
varieties, of which one can take the limit. But, in fact, passing to
the limit commutes with taking joins, and in particular with taking
secant varieties.  In other words, for \pvs{} $\plg$ and $\bY$ the
limit of $\plg + \bY$ as defined in Section~\ref{sec:limit} equals the
closure of the set of all points of the form $x+y$ with $x \in \Xinf$
and $y \in \Yinf$, where the addition takes place in the dual infinite
wedge vector space $(\dirwed)^*$. To see this one uses the right-hand side of
Diagram~\ref{eq:diagX}, which lifts all instances of $X$ to a dense subset
of $\Xinf$. Thus a special case of our main theorem is that {\em higher
secant varieties of Sato's Grassmannian are defined set-theoretically
by finitely many $\Ginf$-orbits of equations}.
Another result that one easily derives from the Noetherianity of matrix
tuples is that, for any number $p$, the Cartesian product of $p$ copies
of Sato's Grassmannian, with the Zariski-topology, is equivariantly
Noetherian with respect to a single copy of $\Ginf$ acting diagonally.
We conclude with a number of open problems.
\begin{enumerate}
\item Is there an ideal-theoretic analogue of our main theorem?  This is
a very interesting, but apparently also very difficult question. It is
{\em not} true that the ideal of the limit of every bounded Pl\"ucker
variety is generated by finitely many $\Ginf$-orbits of equations. Indeed,
using the fact that the ideal of the finite-dimensional Grassmannian
in $\Wedge^p V$ is generated by a number of $\GL(V)$-modules in the
symmetric power $S^2 \Wedge^p V^*$ that is unbounded as $p$ and $\dim
V-p$ grow, one can show that the ideal of Sato's Grassmannian is not
generated by any finite number of $\Ginf$-orbits of polynomials. Thus
any progress on the ideal-theoretic question would require entirely
new ideas. This is different from the situation for ordinary tensors,
where at least a conjectural ideal-theoretic analogue of our theorems
exists \cite{Draisma11d}. A closely related question is how to generalise
Snowden's $\Delta$-modules \cite{Snowden13} to modules of syzygies for
Pl\"ucker varieties. An intermediate question, between our set-theoretic
result and an ideal-theoretic analogue, is whether a variant of our main
theorem holds in the setting of projective schemes; we do not know the
answer to this question either.

\item Most Pl\"ucker varieties of interest to us are constructed from
Grassmannians by operations such as joins and tangential varieties. These
are all bounded. Nevertheless, the restriction to bounded \pvs{} in
Theorems~\ref{thm:main1} and~\ref{thm:main2} seems somewhat {\em ad
hoc}. Are these theorems true for {\em unbounded} Pl\"ucker varieties,
as well? Our proof in Section~\ref{sec:noeth} uses the recursive nature
of Pfaffians in Section~\ref{sec:pfaff} in a fundamental manner. It is
conceivable that, for general Pl\"ucker varieties, this structure can be
replaced with techniques like {\em prolongation} that produce equations
for higher secant varieties given equations for lower secant varieties
\cite{Sidman06,Catalano96}. A second important ingredient in the proofs is the Noetherianity
of matrix tuples, Theorem~\ref{thm:main3}. We may need an analogue of this
for higher-dimensional tensors to deal with general Pl\"ucker varieties.

\item Is the ideal-theoretic version of Theorem~\ref{thm:main3} true? This
question seems easier than the preceding questions, and we conjecture
that the answer is positive. This question is essentially a question
about two-variable tca's in the sense of~\cite{Sam12,SamSnow12}.

\end{enumerate}

%\bibliographystyle{alpha}
%\bibliography{diffeq,draismabook,draismapreprint,draismajournal}

\newcommand{\etalchar}[1]{$^{#1}$}

\end{document}